\documentclass[12pt]{amsart}
\usepackage{amsthm,amsmath,amsxtra,amscd,amssymb,xypic,color,array,lscape}
\usepackage[all]{xy}
\usepackage{arydshln}

\newcommand{\CC}{{\mathbb{C}}}

\newcommand{\HH}{{\mathbb{H}}}

\newcommand{\EE}{{\mathbb{E}}}
\newcommand{\PP}{{\mathbb{P}}}
\newcommand{\QQ}{{\mathbb{Q}}}
\newcommand{\RR}{{\mathbb{R}}}
\newcommand{\ZZ}{{\mathbb{Z}}}

\newcommand{\VV}{{\mathbb{V}}}

\newcommand{\calH}{{\mathcal H}}

\newcommand{\calF}{{\mathcal F}}
\newcommand{\calI}{{\mathcal I}}
\newcommand{\calA}{{\mathcal A}}
\newcommand{\calC}{{\mathcal C}}
\newcommand{\calL}{{\mathcal L}}
\newcommand{\calM}{{\mathcal M}}
\newcommand{\calN}{{\mathcal N}}

\newcommand{\calG}{{\mathcal G}}
\newcommand{\calX}{{\mathcal X}}

\newcommand{\calV}{{\mathcal V}}

\newcommand{\op}{\operatorname}
\newcommand{\ab}[1][g]{\calA_{#1}}
\newcommand{\oab}[1][g]{\overline{\calA}_{#1}}

\newcommand{\Sat}[1][g]{{\calA_{#1}^{\op {Sat}}}}
\newcommand{\Vor}[1][g]{{\calA_{#1}^{\op {Vor}}}}
\newcommand{\Perf}[1][g]{{\calA_{#1}^{\op {Perf}}}}

\newcommand{\Part}[1][g]{{\calA_{#1}^{\op {part}}}}

\newcommand{\Sp}{\op{Sp}}

\newcommand{\GL}{\op{GL}}
\newcommand{\Aut}{\op{Aut}}

\newcommand{\Sym}{\op{Sym}}

\newcommand{\mIH}{{\mathcal I \mathcal H }}

\newcommand\codim{{\rm codim}}

\newcommand\ud{\underline}

\theoremstyle{plain}
\newtheorem{thm}{Theorem}[section]
\newtheorem{lm}[thm]{Lemma}
\newtheorem{prop}[thm]{Proposition}

\theoremstyle{definition}

\newtheorem{rem}[thm]{Remark}

\begin{document}
\title[The intersection cohomology of $\Sat$]{
The intersection cohomology of the Satake compactification of $\calA_g$ for $g \leq 4$}
\author{Samuel Grushevsky}
\address{Mathematics Department, Stony Brook University,
Stony Brook, NY 11794-3651, USA}
\email{sam@math.stonybrook.edu}
\thanks{Research of the first author is supported in part by National Science Foundation under the grants DMS-12-01369 and DMS-15-01265, and by a Simons Fellowship in Mathematics (Simons Foundation grant \#341858 to Samuel Grushevsky)}
\author{Klaus Hulek}
\address{Institut f\"ur Algebraische Geometrie, Leibniz Universit\"at Hannover, Welfengarten 1, 30060 Hannover, Germany}
\email{hulek@math.uni-hannover.de}
\thanks{Research of the second author is supported in part by DFG grant Hu-337/6-2.}

\begin{abstract}
We  completely determine the intersection cohomology of the Satake compactifications $\Sat[2],\Sat[3]$, and $\Sat[4]$, except for $IH^{10}(\Sat[4])$. We also determine all the ingredients appearing in the decomposition theorem applied to the map from a toroidal compactification to the Satake compactification in these genera. As a byproduct we  obtain in addition several results about the intersection cohomology of the link bundles involved.
\end{abstract}
\maketitle
\section{Introduction}

Intersection cohomology is a powerful concept for understanding the topology of algebraic varieties.
A case of particular interest, which has attracted many authors, is to study the intersection cohomology  of compactifications of locally symmetric spaces $Z=\Gamma\backslash D$.
Indeed, intersection cohomology is only one of several possible, highly interesting, cohomology theories which play a role here, and it is exactly the relationship between these cohomology theories
which has been studied extensively.
One highlight is the proof, given independently by  Looijenga~\cite{Lo} and  Saper-Stern~\cite{SaSt}, of Zucker's conjecture stating that the intersection cohomology of the Baily-Borel compactification $Z^{\operatorname{BB}}$ is canonically isomorphic to
the $L^2$-cohomology of $Z$.
Here, and throughout the paper, we always assume that we work with middle perversity.
Goresky, Harder and MacPherson further  introduced the concept of weighted cohomology groups for arithmetic groups in ~\cite{GHM}.
In this paper they also showed that intersection cohomology groups are canonically isomorphic to
certain weighted cohomology groups.
Due to the work of these and other authors, there is a wealth of information on these cohomology theories and their relationship to each other.
In great contrast, however,  very little is known about explicit calculations of these cohomology groups, a fact that is not surprising given the highly complicated construction of the cohomology theories
involved.
Indeed, to our knowledge there are only very few explicit results:
the intersection Euler numbers of the level $n$ covers $\Sat[2](n)$ of the Satake-Baily-Borel compactification   of the moduli space $\ab[2]$ of principally polarized abelian surfaces were computed by Goresky, Harder, MacPherson and Nair for $n\ge 3$~\cite{GHMN} and the
intersection cohomology groups $IH^*(\Sat[2](3))$ of the level $3$ cover were completely determined by Hoffmann and Weintraub~\cite{HW}.

The aim of the current work is to address this question and to present concrete computations for the intersection cohomology of $\Sat$ (with no level structure) for $g\leq 4$.
We can summarize our results from Proposition~\ref{prop:IH2}
together with Theorems~\ref{theo:IH3} and~\ref{theo:IH4} as follows
\begin{thm}
For $g\leq 4$, and for any degree $j$, there is an isomorphism
$$
IH^j(\Sat) \cong R_g^j
$$
between the intersection cohomology of the Satake compactification $\Sat$ and the tautological ring $R_g$ generated by Chern classes $\lambda_i$ of the Hodge bundle --- except that possibly $IH^{10}(\Sat[4])\supsetneq R_4^{10}.$
\end{thm}
The main tool we use is the decomposition theorem for the smooth toroidal resolution $\Vor \to \Sat$ of the Satake-Baily-Borel
compactification given by the second Voronoi compactification in genus $g \leq 4$.
More precisely,  we actually  determine all summands that appear in the decomposition theorem.
The reason why we cannot determine $\dim IH^{10}(\Sat[4])$ is the same as the reason why we do not know all Betti numbers of $\Vor[4]$: the missing information is the
Euler number $e(\ab[4])$, which, amazingly, does not seem to be known.
Furthermore, we obtain the following
\begin{prop}\label{prop:Perf4}
All the odd degree intersection Betti numbers of $\Perf[4]$ are zero, while the even ones
$ib_j:=\dim IH^j(\Perf[4])$ are as follows:
\begin{equation} \label{equ:bettiperf4}
\begin{array}{r|ccccccccccc}
j&0&2&4&6&8&10&12&14&16&18&20\\\hline
ib_j&1&2&4&9&14&*&14&9&4&2&1
\end{array}
\end{equation}
where we know that $*=\dim IH^{10}(\Perf[4])\ge 16$.
\end{prop}

In addition we show that $h^{10}(\Vor[4])\ge 19$, improving the bound $h^{10}(\Vor[4])\ge10$ obtained in~\cite{HT2}.
We also note that, provided the vanishing $H^j(\ab[4] \setminus \overline{\mathcal J}_4,\QQ)=0$ holds for $1 \leq j \leq 9$, where ${\mathcal J}_4$ is the Jacobian locus in genus $4$,
that the singular Betti numbers $b_j(\Perf[4])$ agree with the
numbers $ib_j$ given above for $j\geq 12$ (see~\cite[p. 204]{HT2}).

Along the way we also obtain a wealth of information about the intersection cohomology of the  links involved.

We will denote $\calN_{i,j}$ the link bundle of $\ab[i]$ in $\Sat[j]$ for $i < j$.
This is the quotient of a  real fibration over a level cover  $\ab[i](n)$ by some finite group.
For a detailed discussion of link bundles see Section \ref{sec:level}.
As usual we will identify irreducible  local systems with irredcuible representations of the symplectic group, which in their turn can be identified with Young diagrams.  We will use the standard
notation $\VV_{\mu}$ for the local system corresponding to the irreducible representation of the symplectic group of weight $\mu$, which
we explain in detail in Section \ref{sec:firstlink}.

Then our knowledge about the link cohomology is summarized as follows:
\begin{thm}\label{thm:links}
The intersection cohomology of the links in low genus is as follows:
\begin{equation} \label{equ:linkcoh}
\begin{array}{l|r|rrrrrrrrrrr}
&\dim_\RR\calN&q& 0&1&2&3&4&5&6&7&8&9\\\hline
\rule{0pt}{4mm}    \mIH^q(\calN_{0,2},\QQ)&5&&\QQ&0&0\\
\mIH^q(\calN_{0,3},\QQ)&11&&\QQ&0&0&0&0&0\\
\mIH^q(\calN_{1,3},\QQ)&9&&\QQ&0&0&0\\
\mIH^q(\calN_{0,4},\QQ)&19&&\QQ&0&0&0&0&0&?&0&0&0\\
\mIH^q(\calN_{1,4},\QQ)&17&&\QQ&0&0&0&?&0&?&0&0\\
\mIH^q(\calN_{2,4},\QQ)&13&&\QQ&0&0&0&?&?&?\\
\hline
\rule{0pt}{4mm}    \mIH^q(\calN_{0,3},\VV_{11})&11&&0&0&?&0&0&0\\
\mIH^q(\calN_{1,3},\VV_{11})&9&&?&0&?&0&0\\
\mIH^q(\calN_{2,3},\VV_{11})&5&&?&?&?\\\hline
\end{array}
\end{equation}
where for each link bundle the cohomology above the middle dimension is given by Poincar\'e duality.
For the intersection cohomology labeled $?$ above, we only know its pairwise sums as follows:
$$
\mIH^6(\calN_{0,4},\QQ)\oplus \mIH^2(\calN_{0,3},\VV_{11})=\QQ,
$$
$$
\mIH^4(\calN_{1,4},\QQ)\oplus\mIH^0(\calN_{1,3},\VV_{11})=\VV_2,
$$
$$
\mIH^6(\calN_{1,4},\QQ)\oplus\mIH^2(\calN_{1,3},\VV_{11})=\VV_2.
$$
$$
\mIH^4(\calN_{2,4},\QQ)\oplus\mIH^0(\calN_{2,3},\VV_{11})=\VV_{22},
$$
$$
\mIH^5(\calN_{2,4},\QQ)\oplus\mIH^1(\calN_{2,3},\VV_{11})=\VV_{2},
$$
$$
\mIH^6(\calN_{2,4},\QQ)\oplus\mIH^2(\calN_{2,3},\VV_{11})=\VV_{22},
$$
where each identity is a an equality of local systems on suitable $\ab[i]$: the first line is an identity of local systems (i.e.~vector spaces) on $\ab[0]$, the next two are local systems on $\ab[1]$, and the remaining three are equalities for local systems on $\ab[2]$.
Furthermore, all of the cohomology $\mIH^*(\calN_{g-1,g},\QQ)$ for any $g$ is completely determined in Proposition~\ref{prop:homlink}.
\end{thm}

The proof of this theorem is a combination of a number of individual computations which appear in our analysis of the contributions to the decomposition theorem arising from different strata. We will make  appropriate references throughout the text where these results are proven. The amount of information on link cohomology,  which we could gather from our computations, came as  a surprise to us. In particular, we see no a priori reason for the vanishing we observe. In general, it is very difficult to compute the intersection cohomology of the (singular) links explicitly from first principles.

\medskip
{\small {\bf Acknowledgement.} We would like to thank Mark Goresky and Luca Migliorini very much for generously explaining to us at the Institute for Advanced Study the many details about intersection cohomology and the decomposition theorem. We  are grateful to Eduard Looijenga for numerous enlightening discussions, in particular about the extension of tautological classes to the Satake compactification.
We are especially indebted to Mark Goresky for providing the proof of Proposition~\ref{prop:invariantih} and to Luca Migliorini for detailed comments on a preliminary version of this manuscript. We thank the referee for a careful reading of the manuscript and suggested improvements of the exposition.

Both authors thank the Institute for Advanced Study and the Fund for Mathematics
for support and the excellent working conditions in Spring 2015, when this work was begun.
}

\section{The decomposition theorem}

The main tool for our computation is the decomposition theorem for projective morphisms which is due to Beilinson, Bernstein, Deligne and Gabber --- see~\cite{CM3} for an
excellent exposition.  Throughout this section  let $f: X \to Y$ be a surjective morphism of irreducible  projective varieties.  Let $n$ be the (complex) dimension of $X$ and set
$Y^i:= \{y \in Y, \dim f^{-1}(y)=i \}$. Then the {\em defect} of $f$ is defined as
$$
  r(f):= \max \{2i + \dim Y^i  -n \mid  Y^i \neq \emptyset \}.
$$
The map $f$ is called {\em semi-small} if $r(f)=0$.

The map $f$ defines a stratification $Y= \bigsqcup_{k=0}^{k=\ell} S_{k}$ into locally closed smooth subvarieties $S_{k}$ of increasing dimension $s_k$, such that setting
$X_k:=f^{-1}(S_k)$,
the restricted map $f|_{X_k}: X_k  \to  S_{k}$ is a topological fibration. The sets
$U_{k}:= \bigsqcup_{k' \geq k} S_{k'}$ are dense open subsets of $Y=U_0$ and there are natural inclusions $\alpha_{k}: S_{k}  \subset U_{k}$ where $\alpha_{\ell}$ is the identity map
on $U_{\ell}=S_{\ell}$.
We denote the intersection cohomology complex of a variety $Z$ by $\calI \calC_Z$.
Throughout this section, $\eta$ will denote an ample line bundle on $X$.

The main principles which govern the topology of algebraic maps can then be summarized as follows, see~\cite[Theorem 2.1.1]{CM2} for the case when $X$ is smooth and
\cite[Theorem 1.6.1]{CM3} \label{theo:decomptheorem}for the general case:
\begin{thm}\label{teo:decomposition}
Let $f: X \to Y$ be a surjective projective morphism. Then the following holds:
\begin{itemize}
\item[(1)] (Decomposition Theorem) There is an isomorphism in  the derived category on $Y$:
\begin{equation*}
\varphi: Rf_*\calI \calC_X \cong \bigoplus_i {}^p{\mathcal H}^i(Rf_*\calI \calC_X)[-i]
\end{equation*}
where $i \in [-r(f),r(f)]$.
\item[(2)]  (Relative Hard Lefschetz Theorem) There are isomorphisms
\begin{equation*}
\eta^i:  {}^p{\mathcal H}^{-i}(Rf_*\calI \calC_X) \cong  {}^p{\mathcal H}^i(Rf_*\calI \calC_X).
\end{equation*}
given by the cup product with the powers of the class $\eta$.
\item[(3)] (Semi-simplicity) There are canonical isomorphisms
\begin{equation*}
{}^p{\mathcal H}^i(Rf_*\calI \calC_X) \cong \bigoplus_{k}  \calI \calC_{{\overline S}_{k}}(L_{i,k})
\end{equation*}
where
\begin{equation*}
L_{i,k} =  \alpha_{k}^* \calH^{-s_k}({}^p\calH^i(Rf_*\calI \calC_X))
\end{equation*}
are semi-simple. More precisely, there exist finitely many  indecomposable local systems
$\calL_{i,k,\beta}$ (which can occur repeatedly) such that
\begin{equation*}
 L_{i,k}  \cong \bigoplus_{\beta} \calL_{i,k,\beta}.
\end{equation*}
\end{itemize}
\end{thm}
\begin{rem}
The isomorphism $\varphi$ is not canonical. However the stratification $S_{k}$ and the summands $ \calL_{i,k,\beta}$ are essentially canonical (of course one can always choose
refinements).
\end{rem}

\begin{rem}
An immediate consequence of the decomposition theorem is that we obtain a (non-canonical) isomorphism
\begin{equation}\label{equ:inclusionIH}
IH^m(X,\QQ)  \cong  IH^m(Y,\QQ) \oplus  \bigoplus_{k <\ell, i, \beta}^{} IH^{m- n + s_k + i} (\overline{S}_{k}, {\mathcal L}_{i,k,\beta}).
\end{equation}
\end{rem}

\begin{rem}
In the projective case it is enough for $\eta$ to be $f$-ample for the relative Hard Lefschetz theorem to hold, see~\cite[Theorem 2.3.3]{CM2}.
\end{rem}
\begin{rem}
By the {\em Purity Theorem}~\cite[Theorem 2.2.1]{CM2} the relative Hard Lefschetz isomorphisms respect the decomposition with respect to the strata $S_{k}$.
\end{rem}

An important case is when $X$ is smooth, so that  one can replace $Rf_*\calI \calC_X$ by $Rf_*\QQ_X[n]$ (in the literature this is then often simply written
as $f_*\QQ_X[n]$). In this case $H^*(X,\QQ) \cong IH^*(X,\QQ)$ and the intersection cohomology $IH^*(Y,\QQ)$ becomes a direct
summand of the cohomology $H^*(X,\QQ)$, more precisely there is a (non-canonical) isomorphism
\begin{equation}\label{equ:inclusion}
H^m(X,\QQ)  \cong  IH^m(Y,\QQ) \oplus  \bigoplus_{k <\ell, i, \beta}^{} IH^{m- n + s_k + i} (\overline{S}_{k}, {\mathcal L}_{i,k,\beta}).
\end{equation}
From now on when talking about (intersection) cohomology with $\QQ$ coefficients we will not write the local system, thus writing simply $IH^*(X)$ or $H^*(F)$ for $IH^*(X,\QQ)$ and $H^*(F,\QQ)$ respectively.

\section{The mechanics of the decomposition theorem}\label{sec:mechanics}
In this section we want to explain how one can, in certain cases, inductively determine the local systems which occur in the decomposition theorem.
The process which we describe in this section seems well known to  specialists, but at the same time seems not to be explicitly stated in the literature as such.
It can be extracted from the proof of the decomposition theorem given by de Cataldo and Migliorini via the splitting theorem given in~\cite[\S 4.1]{CM2}.

From now on we consider the special case where $f: X \to Y$  is a {\em birational} morphism of irreducible projective algebraic varieties of dimension $n$, and we further assume $X$ to be smooth. One can then find a Whitney stratification
\begin{equation*}
 Y= S_{\ell} \cup S_{\ell-1} \cup S_{\ell - 2} \cup \ldots \cup S_0
\end{equation*}
such that,  setting $X_k:=f^{-1}(S_k)$, the following holds:
\begin{itemize}
\item[(1)] $S_{\ell}$ is a Zariski open subset of $Y$ and $f|_{X_\ell}: X_{\ell} \to S_{\ell}$ is an isomorphism,
\item[(2)] The $S_k$ are locally closed   smooth subvarieties of increasing dimension $s_k=\dim (S_k) < s_{k+1}$,
\item[(3)] $S_k \subset \overline{S}_{k+1}$,
\item[(4)] For any $k<\ell$ the restriction $f_k:=f|_{X_k}: X_k \to S_k$ is a topological fiber bundle with topological fiber $F_k$.
\end{itemize}
Note that the strata $S_k$ a priori need not be connected and that we do not assume the fibers $F_k$ to be smooth
(but they can be stratified such that the restriction of $f$ to these strata is smooth). In order to avoid a surplus of notation we will assume in this section that the $S_k$ are connected.
It will be obvious how to generalize this to several components. From now on we will always write $F_k$ for the topological type of the fiber of such a fibration, and will write $\calF_k$ for the total space of such a fibration.

For any point $x \in S_k$ we define the link $N_{k,x}$ as the boundary of a small neighborhood in a transversal slice to $S_k$ at $x$.
Since $S_k$ is assumed to be connected, the topological type of $N_{k,x}$ does not depend on $x$, and we will denote this by $N_k$.
Varying $x$ we obtain the {\em link bundle} ${\mathcal N}_k$  over $S_k$.
Finally, taking the intersection cohomology of the  links  ${N}_{k,x}$ we obtain local systems $\mIH^j(\mathcal N_k)$ on $S_k$.
Of course, the local system is not solely determined by its fibers, and in applications it will be important to
know its global structure on $S_k$.

If $k_1 < k_2$,
we can view $S_{k_1}$ as a stratum in ${\overline S}_{k_2}$ and thus define the link bundle ${\mathcal N}_{k_1,k_2}$ with fiber  $N_{k_1,k_2}= N_{k_1,k_2,x}$ being the link of $S_{k_1}$ within
$\overline S_{k_2}$.
In this notation $N_k=N_{k,\ell}$ and ${\mathcal N}_k={\mathcal N}_{k,\ell}$.

Our aim is to understand which local systems occur in the decomposition
\begin{equation}\label{equ:decomp2}
Rf_*\QQ_X = \mathcal I \mathcal C_Y \oplus \bigoplus_{k,\beta,i}\mathcal I \mathcal C_{{\overline S}_k}(\mathcal L_{i,k,\beta})[-i].
\end{equation}

We first recall Deligne's sheaf theoretic construction  of  the intersection complex $\mathcal I \mathcal C_Y $ given by the middle  perversity extension.
For this one starts with the constant local system $\QQ_{S_{\ell}}$ and
then forms
repeated truncated push-forwards. The result is a complex whose cohomology, when restricted to a stratum $S_k$, is a direct sum of irreducible local systems whose fibers are
intersection cohomology groups of the links $N_k$.
We obtain a collection of local systems as depicted in Table~\ref{tab:propQ}. We also note that, as long as $Y$ is unibranched along $S_k$, we can replace $\mIH^0(\calN_k)$ by the constant
system $\QQ$.

\begin{table*}[!hbtp]
\scalebox{0.85}{
$
 \begin{array}{c|c|c|c|c|c}
 n-s_{0}-1& & & &  &\mIH^{n-s_{0}-1}({\calN}_{0})\\
  \vdots&& & & &\vdots\\
n-s_{\ell-2}-1& & &\mIH^{n-s_{\ell-2}-1}({\calN}_{\ell - 2})& \cdots &\mIH^{n-s_{\ell-2}-1}({\calN}_{0})\\
 \vdots&& &\vdots&\vdots&\vdots\\
n-s_{\ell-1}-1& &\mIH^{n-s_{\ell-1}-1}({\calN}_{\ell - 1})&\mIH^{n-s_{\ell-1}-1}({\calN}_{\ell - 2})& \cdots &\mIH^{n-s_{\ell-1}-1}({\calN}_{0})\\
\vdots&&\vdots&\vdots&\vdots&\vdots\\
1 & & \mIH^1({\calN}_{\ell-1})&\mIH^1({\calN}_{\ell-2})&\cdots&\mIH^1({\calN}_{0})\\
0& \QQ&\mIH^0({\calN}_{\ell-1})&\mIH^0({\calN}_{\ell-2})&\cdots&\mIH^0({\calN}_{0})\\
 \hline
 &S_{\ell}&S_{\ell-1}&S_{\ell-2}&\ldots&S_0\\
 \end{array}
 $
 }
 \caption{Contributions of the intersection complex on $Y$}\label{tab:propQ}
 \end{table*}

We now must compare this to $Rf_*\QQ_X$. Since we are in the situation where $f|_{X_{\ell}}: X_{\ell} \cong S_{\ell}$ is an isomorphism it follows that  the only local system on $S_{\ell}$ on the right  hand side  of (\ref{equ:decomp2}) is also the  summand $\QQ_{S_{\ell}}$.
The first step is thus to determine which further local systems  ${\mathcal L}_{i,\ell-1,\beta}[-i]$ occur on $S_{\ell-1}$.
This can be computed, provided one understands the topology of the fibration $f_{\ell-1}: X_{\ell-1} \to S_{\ell-1}$, with fiber $F_{\ell-1}$, sufficiently well; recall that we denote by $\calF_{\ell-1}$ the total space of such a fibration.
More precisely, we have to compute the direct image $\calH^*(\calF_{\ell - 1}):={Rf_{\ell - 1}}_*(\QQ)$ and compare this to the truncated intersection complex $\mIH_Y$ on $S_{\ell -1}$.
Let $\calH^j(\calF_{\ell - 1}):={R^jf_{\ell - 1}}_*(\QQ)$ be the local system given by taking $j$-th cohomology of the fibers.
For $j \leq n - s_{\ell -1} -1$ each local system $\mIH^j({\calN}_{\ell-1})$ is a summand of $\calH^j(\calF_{\ell - 1})$.
The irreducible local systems $\mathcal L_{i,\ell -1,\beta}[-i]$ in the decomposition theorem  (\ref{equ:decomp2}) in degree $j$, which are supported on $S_{\ell - 1}$,  are then all those irreducible
summands of $\calH^j(\calF_{\ell - 1})$ which are not accounted
for by the  truncated collection of  local systems $\mIH^j({\calN}_{\ell-1})$ for $j \leq n - s_{\ell -1} -1$. This includes, in particular, all irreducible summands of
the local systems $\calH^j(\calF_{\ell - 1}), j > n - s_{\ell -1} -1$. The shift $[-i]$ in (\ref{equ:decomp2}) of a new local system appearing in degree $j$, is given by  $[-i]=-j + \codim(S_{\ell - 1},Y)=- j + n - s_{\ell-1}$. We will call the local systems which are not accounted for by the truncated complex $\mIH_Y$ the {\em new} local systems on $S_{\ell -1}$.

It is now clear how to proceed inductively. In the previous steps we have accounted for the entire fiber cohomology $\calH^*(\calF_{\ell -1})$, as local systems over $S_{\ell-1}$, either by contributions from $\mIH_Y$
or by {\em new} local systems. The next step is to study the fiber cohomology $\calH^*(\calF_{\ell-2})$ as a collection of local systems on $S_{\ell-2}$. Both  $\mIH_Y$ and the newly found systems on $S_{\ell-1}$ will
contribute to this. To explain the contribution of the new local systems, we consider a new local system $\calG$ on $S_{\ell-1}$ that is a summand of $\calH^j(\calF_{\ell-1})$ not accounted for by  $\mIH_Y$. This contributes to the total cohomology of $X$ via the intersection cohomology complex $IH^*({\overline S}_{\ell-1}, \calG)$, shifted by  $[-i]=-j + n - s_{\ell-1}$.
This means that we  have to form the intersection cohomology complex
$\calI\calC_{{\overline S}_{\ell-1}}(\calG)$, given by repeated truncated pushforwards and shifted appropriately. Table~\ref{tab:propQ2} shows the contributions which we obtain in this way on the strata $S_k$ for $k \leq \ell-1$.
In general, at each inductive step we thus have to compare $\calH^*(\calF_{k})$ with all the contributions which we have in this way obtained from the previous strata $S_{r}, r>k$. The difference gives us the new local systems that live on $S_k$.

 \begin{table*}[!hbtp]
\scalebox{0.62}{
$
 \begin{array}{c|c|c|c|c|c}
 j+s_{\ell-1}-s_0 - 1& & & &  &\mIH^{ s_{\ell-1}-s_0 - 1}(\calN_{0,\ell-1},\calG)\\
  \vdots&& & & &\vdots\\
j+s_{\ell-1}-s_{\ell-3}-1& & &\mIH^{s_{\ell-1}-s_{\ell-3}-1}(\calN_{\ell - 3,\ell-1},\calG)& \cdots &\mIH^{s_{\ell-1}-s_{\ell-3}-1}(\calN_{0,\ell-1},\calG)\\
 \vdots&& &\vdots&\vdots&\vdots\\
j+s_{\ell-1}-s_{\ell-2}-1& &\mIH^{s_{\ell-1}-s_{\ell-2}-1}(\calN_{\ell - 2,\ell-1},\calG)&\mIH^{s_{\ell-1}-s_{\ell-2}-1}(\calN_{\ell - 3,\ell-1},\calG)
& \cdots &\mIH^{s_{\ell-1}-s_{\ell-2}-1}(\calN_{0,\ell-1},\calG)\\
\vdots&&\vdots&\vdots&\vdots&\vdots\\
j& \calG&\mIH^0(\calN_{\ell-2,\ell-1},\calG)&\mIH^0(\calN_{\ell-3,\ell-1},\calG)&\cdots&\mIH^0(\calN_{0,\ell-1},\calG)\\
 \vdots&&&&&\\
 \hline
 &S_{\ell-1}&S_{\ell-2}&S_{\ell-3}&\ldots&S_0\\
  \end{array}
 $
 }
 \caption{Contributions of the new local system $\calG$}\label{tab:propQ2}
\end{table*}

Finally, one can reformulate the relative Hard Lefschetz theorem in this situation. Recall that this implies a bijection between the local systems ${\mathcal L}_{i,k,\beta}[-i]$ and ${\mathcal L}_{-i,k,\beta}[i]$ for $1 \leq i \leq r(f)$. We can reformulate this as follows.

\begin{prop}\label{prop:symmetry}
Let $f_k:=\dim F_k$. For every new irreducible local system arising from $\calH^{n-s_k+i}(\calF_k)$ for $1 \leq i \leq 2f_k+s_k-n$, we find a copy of the same local system in $\calH^{n-s_k-i}(\calF_k)$, which is new, i.e.~not accounted for by the local systems coming from contributions from $S_{r}$, $r > k$.
\end{prop}

This can be viewed as a reflective symmetry of the set of the new local systems appearing over $S_k$, where the reflection axis is given by the complex codimension $\codim(S_k,Y)=n-s_k$.

\section{The set-up for $\Sat$}\label{sec:setup}
We will now specialize the set-up of the decomposition theorem to the situation for the moduli space $\ab$ of complex principally polarized abelian varieties of dimension $g$.
Recall that the Satake (or Baily-Borel) compactification $\Sat$ is a projective variety.
Every admissible cone decomposition $\Sigma$ of the rational closure of the cone $\Sym^2_{>0}(\RR^g)$ of positive symmetric $g \times g$ matrices defines a
toroidal compactification ${\mathcal A}_g^{\Sigma}$ which admits a contracting morphism $\phi_g^{\Sigma}: \ab^{\Sigma} \to \Sat$.
In this paper we are primarily concerned with the topology of projective varieties.
Nevertheless, it should be borne in mind that $\ab$ and its compactifications represent corresponding moduli stacks.
These are Deligne-Mumford stacks and the only stack property we make use of is that $\ab$, as well as its compactifications  $ \ab^{\Sigma} $,
can be written as a finite quotient of the level $n$ covers by the deck
transformation group $\Sp(2g,\ZZ/n\ZZ)$.
We can always choose $\Sigma$ such that the level covers  $\ab^{\Sigma} (n)$ are projective and smooth for $n\geq 3$. In this, and analogous cases, we use the terminology {\em stack smooth}
for (strata of) moduli spaces of abelian varieties which are the quotient of smooth varieties by a finite group action.
We will exploit this in two ways. One is that in order to compute the cohomology of $\ab^{\Sigma} $ we can compute the cohomology of the manifold  $\ab^{\Sigma}(n) $ and then take
invariant cohomology. The other is that we can  use these level covers to put ourselves  into a situation where  we can apply the decomposition theorem and the techniques of
Section~\ref{sec:mechanics} to the map $\phi_g^{\Sigma}(n): \ab^{\Sigma}(n) \to \Sat(n)$ using the  stratification induced by
\begin{equation}\label{equ:beta}
  \Sat = \ab \sqcup \ab[g-1] \sqcup \ldots \sqcup \ab[0].
\end{equation}
This, however, also means that we will have to work with invariant cohomology. We will explain this in detail in Section \ref{sec:level}.

In order to perform our computations we will make  use of two specific toroidal compactifications, namely the
{\em perfect cone} compactification $\Perf$ and the {\em second Voronoi} compactification $\Vor$, both of which are projective.
We recall that for $g \leq 3$ all known toroidal compactifications in a given genus are equal (if one does not take blow-ups obtained by subdivisions of the fans into consideration).
Hence we shall simply denote these compactifications by $\oab$ for $g\leq 3$ and we recall that these are stack smooth.
In genus 4 the central cone compactification, which is the Igusa compactification, is the same as $\Perf[4]$, and has a unique (stack) singular point. The second Voronoi toroidal compactification  $\Vor[4]$ is stack smooth  and is obtained by blowing up the singular point of $\Perf[4]$, see~\cite{HT2} for details. Thus we can apply the setup of Section~\ref{sec:mechanics}  to $\Vor[4]$, and to unify notation we set
$\oab[4]:=\Vor[4]$.

In order to make use of the relative hard Lefschetz,  we need  a relatively ample line bundle for the resolution  $\phi_g^{\Sigma}: \ab^{\Sigma} \to \Sat$, which will play the role
of $\eta$ in the decomposition theorem. Recall that for
any $g\ge 2$, the Picard group of $\Perf$ has two generators (over $\QQ$), namely the Hodge line bundle $L$ and the boundary divisor $D$. The Hodge line bundle $L$ is the pullback
of the Hodge line bundle $L$ (we will use the same letter for both bundles by abuse of notation) on $\Sat$, where it is an ample line bundle. The divisor $-D$ is relatively ample for
$\phi_g^{\operatorname {Perf}}: \Perf \to \Sat$. We will use this for $g=2,3$. In genus $4$ we recall that the (rational) Picard group of $\Vor[4]$ is generated by three elements, namely the
Hodge line bundle $L$ together with two divisors  $D^{\operatorname{Vor}}$ and $E$, where $E$ is the exceptional divisor of $\Vor[4] \to \Perf[4]$. If $\pi: \Vor[4] \to \Perf[4]$  is the blow-up map, and $D$ is the boundary on $\Perf[4]$, then
$\pi^*(D)=D^{\operatorname{Vor}}+4E$ on $\Vor[4]$.
In this case we can take $-2D^{\operatorname{Vor}}-E$ as a relatively ample line bundle
for $\phi_4^{\operatorname {Vor}}: \Vor[4] \to \Sat[4]$.
For details we refer the reader to~\cite[Theorem I.8]{HS}.

The preimages $\beta_i^{\Sigma,0}:=(\phi^{\Sigma}_g)^{-1}(\ab[g-i])$
can be further stratified as $\beta_i^{\Sigma,0}= \sqcup_{\sigma} \beta(\sigma)$ where  $\sigma$ runs through the $\GL(i,\ZZ)$-orbits
of the  decomposition $\Sigma$ in dimension $i$ containing rank $i$ matrices. The individual strata $\beta(\sigma)$ are stack smooth
and the restriction $\phi^{\Sigma}_g|_{\beta(\sigma)}: \beta(\sigma) \to \ab[g-i]$ is, up to taking finite quotients, a submersion (more precisely, the level $n$
covers $\phi^{\Sigma}_g(n)|_{\beta(\sigma)(n)}: \beta(\sigma)(n) \to \ab[g-i](n)$ are submersions for $n \geq 3$).
In our application of the decomposition theorem the $\ab[k]$ will take on the role of the strata $Y_k$ and the $X_k$
will be the $\beta_k^0$. Strictly speaking, we will apply the decomposition theorem to suitable level covers and then  take invariant (intersection) cohomology. We will explain the technical details of this in the next section.

The most natural cohomology classes on $\ab$ are the Chern classes $\lambda_i: =c_i(\mathbb E)\in H^{2i}(\ab,\QQ)$, $i=1, \ldots, g$ of the Hodge vector bundle $\mathbb E$.
Although the Hodge bundle does not extend to the Satake compactification $\Sat$, Charney and Lee~\cite{ChLe}
have constructed lifts of these classes to $H^{2i}(\Sat,\QQ)$. These lifts are not canonically defined, but we will choose a compatible sequence of such lifts once and for all, which, by abuse of notation, we will also denote by $\lambda_i \in H^{2i}(\Sat,\QQ)$. This choice will be irrelevant for our purposes.
We remark that another lift of the classes $\lambda_i$ to $\Sat$ was constructed by Goresky and Pardon~\cite{GP}, but their lifts a priori only live in cohomology $H^{2i}(\Sat,\CC)$
with complex coefficients. In fact, by a result of Looijenga~\cite{Lo2} the Goresky-Pardon classes have a non-trivial imaginary part and thus do not live in  $H^{2i}(\Sat,\QQ)$. Since, however, we prefer
to work with rational coefficients, we will work with Charney-Lee classes.

Given a toroidal compactification  $\ab^{\Sigma}$ we can pull the classes $\lambda_i \in H^{2i}(\Sat,\QQ)$ back to $H^{2i}(\ab^{\Sigma},\QQ)$ and once more by abuse of notation, we will again
denote these classes by $\lambda_i$. These classes satisfy the relation
\begin{equation} \label{equ:relation}
(1-\lambda_1 + \lambda_2 - \ldots +(-1)^g\lambda_g)(1+ \lambda_1 + \lambda_2 + \ldots +  \lambda_g)=0
\end{equation}
and in fact there is no other relation between these classes. This relation follows from the fact that the Hopf algebra constructed by Charney and Lee~\cite{ChLe} has no even  Chern
characters as its primitive generators. Note that the Goresky-Pardon classes, when pulled back to a (stack) smooth toroidal compactification, become the
Chern classes of the extended
Hodge bundle, namely  $\lambda_i=c_i(\EE)$ and relation (\ref{equ:relation})  is well known to hold both in cohomology (where we will work) and also in Chow~\cite{vdG2,EV}.

The {\em tautological ring} is  defined as the subring of $H^*(\ab^\Sigma)$ generated by the $\lambda$ classes:
\begin{equation}
R_g= \QQ[\lambda_1, \ldots , \lambda_g].
\end{equation}
Using the fact that (\ref{equ:relation}) is the only relation it follows that the tautological ring is zero in degree above $g(g+1)$, is one-dimensional in degree $g(g+1)$, and satisfies a perfect pairing with socle in dimension $g(g+1)$. In fact, a basis of the tautological ring as a vector space is given by the $2^g$ polynomials $\prod_{i=1}^g \lambda_i^{\varepsilon_i}$, for all $\varepsilon_i\in\lbrace 0,1\rbrace$, with the obvious pairing $\varepsilon_i\mapsto 1-\varepsilon_i$.

Recall that for any variety there exist maps $H^*(X) \to IH^*(X)$. Moreover one can multiply cohomology classes with intersection cohomology classes, and
the product $IH^*(X) \otimes H^*(X) \to IH^*(X)$ makes $IH^*(X)$ into a module over $H^*(X)$. In particular it makes sense to consider the
$\lambda$-classes and their products  as elements in $IH^*(\Sat)$.

Our first observation is
\begin{prop}\label{prop:tautological}
There is an inclusion $R_g \hookrightarrow IH^*(\Sat)$ of the tautological ring $R_g$ into intersection cohomology of $\Sat$.
\end{prop}
\begin{proof}
By the map from cohomology to intersection cohomology the classes $\lambda_i$ and their products can be mapped to the intersection cohomology $IH^*(\Sat)$ preserving
relation (\ref{equ:relation}). We claim that the map $R_g \to IH^*(\Sat)$ is indeed an inclusion. Indeed, this follows from the fact that the intersection product defines a perfect pairing on $R_g$,
which shows that the only relations among the $\lambda$-classes in $IH^*(\Sat) \subset H^*(\ab^{\Sigma})$ are given by  (\ref{equ:relation}).
\end{proof}
As a consequence of the Proposition we can, and will,  regard $R_g$ as contained in both $H^*(\ab^{\Sigma})$ and $IH^*(\Sat)$.

\section{Level covers: an intermediate step}\label{sec:level}
The principal idea is to apply the decomposition theorem to a resolution $\phi_g^{\Sigma}: \ab^{\Sigma} \to \Sat$ and to use
the stratification of the base $\Sat$ given by
\begin{equation*}
  \Sat = \ab \sqcup \ab[g-1] \sqcup \ldots \sqcup \ab[0].
\end{equation*}
The main technical problem with this approach is the non-neatness of the group  $\Sp(2g,\ZZ)$, i.e.~the existence of elements with a non-trivial fixed point set. As a result $\ab^{\Sigma}$ will always have at least orbifold
singularities, and the projections $\beta_i^{\Sigma,0}:=(\phi^{\Sigma}_g)^{-1}(\ab[g-i]) \to \ab[g-i]$ will not be topological fibrations. Both problems can be solved by working with suitable level covers and taking invariant cohomology.

In what follows we will make essential use of the following commutative diagram
\begin{equation*}
\xymatrix@C=66pt
{
\ab^{\Sigma}(n)\ar[r]^{\phi_g^{\Sigma}(n)}
\ar[d]_{p_n^{\Sigma}}
&\Sat(n) \ar[d]^{p_n}\\
{\ab^{\Sigma}}\ar[r]^{\phi_g^{\Sigma}}&\Sat
}
\end{equation*}
where the vertical maps are Galois covers with Galois group $G_g(n)=\Sp(2g,\ZZ/n\ZZ)$ and  the horizontal maps are the projections from a toroidal
compactification to the Satake compactification.

We will first describe the geometry of this diagram.
Under the covering map
\begin{equation*}
p_n: \Sat(n) \to \Sat= \ab \sqcup \ab[g-1] \sqcup \ldots \sqcup \ab[0]
\end{equation*}
the inverse image $p_n^{-1}(\ab)=\ab(n)$ is irreducible, but the other strata decompose into connected components
\begin{equation}
p_n^{-1}(\ab[i])= \bigsqcup_{j=1, \ldots, N_{i,n}}Ê\ab[i]^{(j)}(n)
\end{equation}
where the numbers $N_{i,n}$ can, in principle, be computed explicitly, cf.~\cite[Chapter III]{Er}, and the $\ab[i]^{(j)}(n)$ are all copies of $\ab[i](n)$ .

We have already pointed out that, due to elements with a non-trivial fixed point set,
the contraction morphisms $\phi_{g,i}^{\Sigma}=\phi_g^{\Sigma}|_{\beta_i^{\Sigma,0}}: \beta_i^{\Sigma,0} \to \ab[g-i]$
fail to be topological fibre bundles. The situation, however, improves on the level covers.
Restricting the contraction morphism $\phi_g^{\Sigma}(n)$
to preimages of the components $\ab[g-i]^{(j)}(n)$  gives us topological fibrations
\begin{equation}\label{equ:strat}
\phi_{g,i,n}^{\Sigma,{j}}:= \phi_{g}^{\Sigma}(n)|_{ \beta^{\Sigma ,0,j}_i(n)}: \beta^{\Sigma ,0,j}_i(n) \to  \ab[g-i]^{(j)}(n).
\end{equation}
We recall that
\begin{equation}
\beta^{\Sigma ,0,j}_i(n)= \bigsqcup_{\sigma} \beta^{\Sigma ,0,j}_i(\sigma)(n)
\end{equation}
where $\sigma$ runs through all orbits with respect to the level $n$ subgroup of $\GL(i,\ZZ)$ of cones in $\Sigma$ in dimension $i$ containing matrices of rank $i$.
For any $n\geq 3$ the fibrations
\begin{equation}\label{equ:stratsigma}
\phi_{g,i,n}^{\Sigma,{j}}(\sigma): \beta^{\Sigma ,0,j}_i(\sigma)(n) \to \ab[ g-i]^{(j)}(n)=\ab[g-i](n)
\end{equation}
have the structure of a torus bundle over the $i$-fold fiber
product ${\mathcal X^{\times i}_{g-i}}(n) \to \ab[g-i](n)$. The rank of the torus fiber of $\phi_{g,i,n}^{\Sigma,{j}}(\sigma)$  is equal to $i(i+1)/2 - \dim(\sigma)$, and the fibration can be described explicitly as explained in~\cite[Section 7]{GHT}. This fibration is independent of $j$, and we will thus drop the superscript $j$ from all the related notation, whenever it is clear that we are working on a fixed component.

For future use it is useful to analyze the action of the deck transformation group $G_g(n)$ more closely.
An element $g\in G_g(n)$ either maps a component
$\ab[i]^{(j)}(n)$ to another (disjoint) component $\ab[i]^{(j')}(n)$ or to itself. Let $G_{g,i}(n)=G_{g,i}^{j}(n)$ be the stabilizer of a component $\ab[g-i]^{(j)}(n)$.
This group  acts on the fibration (\ref{equ:stratsigma}), and it can be described explicitly in terms of generators, as discussed in~\cite[Section 7]{GHT}, where these generators are denoted $g_i$, and we now recall what they are.

Let $K_{g,i}(n)$ be the kernel of $G_{g,i}(n) \to \Aut(\ab[g-i](n))$, in other words the subgroup acting on the fibers of (\ref{equ:stratsigma}).
The quotient $G_{g,i}(n) /K_{g,i}(n) \cong \Sp(2(g-i),\ZZ/n\ZZ)$ then acts on $\ab[g-i](n)$ with quotient $\ab[g-i]$.
The action of $G_{g,i}(n)$ can be described explicitly in terms of the generators given in~\cite[Section 7]{GHT}, see e.g.~\cite[Section 5]{HT2}.
Elements of the form $g_1$ and $g_3$ act on the torus bundles
$\beta^{\Sigma ,0}_i(\sigma)(n)$ either by multiplication with $n$-th roots of unity on the tori or by translation by $n$-torsion points on the universal abelian varieties.
In particular, these elements act trivially on the cohomology of the fibers. For the elements $g_4$ we have two possibilities: either such an element maps a cone
$\sigma$ to another cone $\sigma'$, in which case it also maps $ \beta^{\Sigma ,0}_i(\sigma)(n)$ to $ \beta^{\Sigma ,0}_i(\sigma')(n)$ or it maps $\sigma$, and
correspondingly also $ \beta^{\Sigma ,0}_i(\sigma)(n)$, to itself.
Finally, the elements $g_2$ act trivially on the fibers of  (\ref{equ:strat}) and generate $\Sp(2(g-i),\ZZ/n\ZZ)$.

At this point it is useful to recall a basic fact about the cohomology of local systems on $\ab$. When we speak about local systems on $\ab$ we mean, strictly speaking, local systems of
the stack $\ab$ or, in more down to earth terms, $\QQ$-local systems or orbifold local systems. More precisely, a local system on $\ab$, or, equivalently, for the group $\Sp(2g,\ZZ)$, is
given by a representation $\rho: \Sp(2g,\ZZ) \to \Aut(V)$, where $V$ is a finite dimensional vector space over $\QQ$.  Geometrically, the local system is then given by the
quotient $\calV= \Sp(2g,\ZZ) \backslash \HH_g \times V$. If $-1\in \Sp(2g,\ZZ)$ does not act trivially on $V$, or at points in $\HH_g$ which have non-trivial stabilizer, the quotient
$\calV$ is not a local system in the sense of a locally trivial vector bundle with fiber a rational vector space. We can, however, think of $\calV$ as a local system on
the level covers $\ab(n)$ for $n\geq 3$. The cohomology groups  $H^j(\ab,\calV)$ are then defined as the $G_g(n)$-invariant part $H^j(\ab(n),\calV)^{G_g(n)}$ of the cohomology
of $\calV$ on $\ab(n)$. To show that this result does not depend on the chosen level $n$, one uses  the level cover  $\ab(nm)$, which
is a Galois cover of both $\ab(n)$ and  $\ab(m)$. Independence then follows from pulling back to  $\ab(nm)$ and taking invariant cohomology.
A similar construction is also used for the intersection cohomology groups $IH^j(\Sat,\calV)$, which we also define as the invariant intersection cohomology groups
$IH^j(\Sat(n),\calV)^{G_g(n)}$. The independence of the given level then follows as above, using Proposition~\ref{prop:invariantih} below.

There are two types of local systems which will be instrumental to our computations. To explain the first type we have to go back to the fibrations given by (\ref{equ:strat}).
These define local systems
\begin{equation}\label{equ:localsys}
\calH^q(\calF_{g,i,n}^{\Sigma,{j}}) := (R^q\phi_{g,i,n}^{\Sigma,{j}})_*(\QQ)
\end{equation}
on $\ab[i]^{(j)}(n)=\ab[i](n)$.  These local systems can be computed from the corresponding local systems $\calH^q(\calF_{g,i,n}^{\Sigma,{j}} (\sigma))$
for the submersions (\ref{equ:stratsigma}), which can be determined using the torus bundle structure of these fibrations.
Fixing again a component and thus an index $j$,  we will now again drop the superscript $j$ from the notation.
Using the stratification of $\beta^{\Sigma ,0}_i(n)$, given by
orbits of the cones $\sigma$, one can construct a Gysin spectral sequence of local systems (with cohomology of compact support) whose $E_{\infty}$ terms are the local systems
$\calH^q(\calF_{g,i,n}^{\Sigma})$. For this, we note that all strata $\beta^{\Sigma ,0}_i(\sigma)(n)$ are smooth so that we can use Poincar\'e duality.
This technique was also used in ~\cite[Section 9]{GHT} as well as~\cite{HT} and~\cite{HT2}. The difference to these papers is that in our case we apply the Gysin sequence fiberwise in order to
obtain the fiber cohomology.
By construction, the group $G_{g,i}(n)$ acts on the local systems  $\calH^q(\calF_{g,i,n}^{\Sigma})$. The subgroup $K_{g,i}(n)$ acts fiberwise and we can form the invariant
subsystem, which we denote
$(\calH^q(\calF_{g,i,n}^{\Sigma}))^{K_{g,i}(n)}$.
The group $G_{g,i}(n) /K_{g,i}(n)$ now acts on $(\calH^q(\calF_{g,i,n}^{\Sigma}))^{K_{g,i}(n)}$ and this defines a local system
$\calH^q(\calF_{g,i}^{\Sigma})$ for $\Sp(2(g-i),\ZZ)$ on $\ab[g-i]$. By comparing different level structures, one sees again that this does not depend on $n$, and thus we have dropped this index, too.
It now follows from Proposition~\ref{prop:invariantih} that
\begin{equation}\label{equ_equivariant}
IH^*(\Sat[g-i],\calH^q(\calF_{g,i}^{\Sigma})) \cong IH^*(\Sat[g-i](n),\calH^q(\calF_{g,i,n}^{\Sigma}))^{G_{g,i}(n)}.
\end{equation}
We will make frequent use of this fact, often without stating this explicitly.

The second type of local systems which we will use in our application of the decomposition theorem comes from the link bundles $\calN_{g-i}(n)=\calN_{g-i}^{j}(n)$ of $\ab[g-i]^{(j)}(n)$ in $\Sat(n)$.  Taking fiberwise intersection cohomology we obtain local systems $\mIH^q(\calN_{g-i}(n))$. Again, we can take the invariant system with respect to  $K_{g,i}(n)$, which then descends to a local system on $\ab[g-i]$ which we will denote by  $\mIH^q(\calN_{g-i})$. Similarly, if we consider specifically $\ab[k]$ in $\Sat[m]$ we thus obtain linear systems $\mIH^q(\calN_{k,m})$. Note that we do not claim that this is the local system associated to the links of $\ab[k]$ in $\Sat[m]$. The latter is not even well defined because the topology of the inclusion $\ab[k] \subset \Sat[m]$ is not constant at points of $\ab[k]$ that correspond to abelian varieties with extra automorphisms.

The following Proposition is fundamental for our calculations. As we did not find it explicitly in the literature, we will state and prove it here for the sake of completeness. Let $X$ be an irreducible projective variety and $\calL$ be a local system on an open subset $U$ of $X$. We further assume that there is a finite group $G$ acting on $X$ and $\calL$, in particular we assume that $G$ maps $U$ to itself. We denote $Y=X/G$ and $V=U/G$ and by $f: XÊ\to Y$ the quotient map. Let $K$ be the subgroup of $H$ which acts trivially on $X$. Then the $K$-invariant part $\calL^K$ of $\calL$ descends to $Y$, in other words  there is a local system $\calM$ on $V$ with $f^*(\calM)=\calL$. This allows us to formulate the following proposition, the proof of which was supplied by Mark Goresky.

\begin{prop}\label{prop:invariantih}
For any $j$, there is an isomorphism
\begin{equation*}
IH^j(Y,\calM) \cong IH^j(X,\calL)^G.
 \end{equation*}
\end{prop}
\begin{proof}
We first recall from~\cite{GM} that the map which associates to a local system $\calL$ the intersection chain complex
$\calI\calC_X\calL$ defines an equivalence of categories between local systems on $U$ and certain complexes of sheaves on $X$, which satisfy the axioms of intersection cohomology theory.
We shall use this and the analogous statement for $V$ and $Y$. At the same time this allows us to shrink the open set $U$ and we can thus assume that $G/K$ acts freely on $U$.

We recall that the complex $Rf_*(\calL)$ is defined by replacing $\calL$ by an injective resolution, to which $f_*$ is applied. The group $G$ acts on $Rf_*(\calL)$, and taking invariant parts
we obtain a complex, which we denote $Rf_*(\calL)^{RG}$. In our situation it is easy to describe $Rf_*(\calL)^{RG}$ explicitly, up to quasi-isomorphisms.
Since $f: X \to Y$ has finite fibers, all higher cohomology groups
vanish on the fibers. Thus $Rf_*(\calL)$ is quasi-isomorphic to the local system $f_*\calL$ whose stalk at a point $y$ is given by
\begin{equation}\label{equ:pushforward}
(f_*\calL)_y= \bigoplus_{x \in f^{-1}(y)} \calL_x.
\end{equation}
From this we obtain $Rf_*(\calL)^{RG}\cong (f_*\calL)^G$.
The $G$-invariant part in each stalk of $f_*\calL$ is the equivariantly embedded stalk of $\calM$. Hence we obtain on $V$ the following isomorphisms in the derived category
\begin{equation}\label{equ:inv}
\calM  \cong (f_*\calL)^G \cong Rf_*(\calL)^{RG}  .
\end{equation}

We will now progress to intersection cohomology. From (\ref{equ:inv}), and using the equivalence between local systems on $V$ and intersection chain complexes on $Y$ we obtain
\begin{equation}\label{equ:1}
\calI\calC_Y\calM \cong \calI\calC_Y(f_*\calL)^G.
\end{equation}
Moreover, using again the equivalence of categories, we have an isomorphism (of abelian groups)
\begin{equation*}
\operatorname{Hom}(\calI\calC_Y f_*\calL,\calI\calC_Y f_*\calL) \cong \operatorname{Hom}(f_*\calL,f_*\calL)
\end{equation*}
and the action of $G$ on $f_*\calL$ and its decomposition into $G$-irreducible local systems corresponds exactly to the action of $G$ on $\calI\calC_Y f_*\calL$ and its decomposition into
$G$-irreducible complexes of sheaves. In particular, taking invariants, we obtain
\begin{equation}\label{equ:2}
(\calI\calC_Y f_*\calL)^{RG} \cong \calI\calC_Y(f_*\calL)^G.
\end{equation}

Clearly, $\calI\calC_X\calL$ satisfies the properties of intersection cohomology and, since $f$ is finite, the same holds for $Rf_*(\calI\calC_X\calL)$. Hence this is an intersection cohomology sheaf
and is itself determined by its restriction to $V$, which is $Rf_*\calL\cong f_*\calL$. The latter, as we have already seen, follows since $f$ is finite. This implies
\begin{equation*} \label{equ:3}
Rf_*(\calI\calC_X\calL) \cong \calI\calC_Y f_*\calL.
\end{equation*}
Restricting this to  $G$-invariants and further using (\ref{equ:1}) and (\ref{equ:2}) we finally obtain the following isomorphism in the derived category
\begin{equation*}
\calI\calC_Y\calM \cong Rf_*(\calI\calC_X\calL)^{RG}.
\end{equation*}
Taking cohomology  gives the claim.
\end{proof}
The results of this section finally allow us to work essentially on $\Sat$ itself, rather than on level covers. This will also be the way we present our computations. The reduction to level covers is an essential step in our computations, as it allows us to apply the decomposition theorem  for the case when the source of the map is smooth. From now on we will use this reduction to level covers implicitly and not mention it specifically in each step.

\section{The first step: the link bundle $\calN_{g-1,g}$}\label{sec:firstlink}
Before proceeding to work out the low genus cases in complete detail, we demonstrate explicitly how the above machinery works in the first step of the recursive procedure, that is over $\ab[g-1]\subset\Sat$. This is to say, we now work with Mumford's partial toroidal compactification $\Part$ --- the preimage of $\ab\sqcup\ab[g-1]\subset\Sat$ in $\ab^\Sigma$.
Note that we do not need projective varieties to apply the decomposition theorem, it is enough that $\Part \to \ab\sqcup\ab[g-1]$ is a projective morphism.

We first recall the geometry involved here. Over $\ab[g-1]$ all toroidal compactifications $\ab^{\Sigma}$ coincide, and
$\phi_g^{\Sigma}|_{\beta_1^{\Sigma,0}}: \beta_1^{\Sigma,0} \to \ab[g-1]$ is the universal Kummer variety.
For $n\geq 3$ the fibration $\phi_g^{\Sigma,j}(n)|_{\beta_1^{\Sigma,0,j}(n)} : \beta_1^{\Sigma,0,j}(n) \to \ab[g-1]^{(j)}(n)$
is the universal abelian variety. As this fibration depends neither on $j$ nor on $\Sigma$ we will again drop these from the notation. It will now be more convenient to use notation closer to Section~\ref{sec:mechanics},  and we will thus relabel this fibration as  $f_{g-1,g}(n): \calF_{g-1,g}(n) \to \ab[g-1](n)$, with topological fiber $F_{g-1,g}$ and similarly for the other strata  $f_{g-k,g}(n): \calF_{g-k,g}(n) \to \ab[g-k](n)$, with topological fiber $F_{g-k,g}$. Whenever
$g$ is fixed we will further simplify the notation by writing $F_{g-k}=F_{g-k,g}$, and will drop the level $n$ when working without level structure. Similarly, we will write $\calH^q(\calF_{g-k})=\calH^q(\calF_{g-k,g})$ and use $\mIH^q(\calN_{g-i})$, respectively $\mIH^q(\calN_{k,m})$, which were introduced in Section~\ref{sec:level}. We also recall
that for the decomposition theorem only the intersection cohomology of the links below their middle dimension plays a role.

The partial compactification $\Part(n)$ is smooth for $n\ge 3$
and the link $N_{g-1,g}(n)$ is simply the boundary of a slice  of a neighborhood of a fiber of $F_{g-1,g}(n)$, transversal to the boundary stratum, within $\Part(n)$. Globally, such neighborhoods form a tubular neighborhood of a complex codimension one subvariety of $\Part(n)$, and thus the link bundle $\calN_{g-1,g}(n)$ is an $S^1$ circle bundle over $\calF_{g-1,g}(n)=\mathcal X_{g-1}(n)$.
We note that this link is smooth and
hence we can replace its intersection cohomology by ordinary cohomology. We also note that the cohomology of an $S^1$ circle bundle over the universal abelian variety can be
computed using the Leray spectral sequence.

At this point we must  recall further basic facts about irreducible local systems on $\ab$.
These correspond to irreducible representations of $\Sp(2g,\ZZ)$, which in turn correspond  to Young diagrams   $\ud\mu = (\mu_1\geq \mu_2\geq\dots\geq \mu_g)$ with at most $g$ rows.
Equivalently, we can view $\ud\mu$ as an arbitrary partition of length at most $g$. If we denote by $\VV_1$ the standard rational representation of the group  $\Sp(2g,\ZZ)$ on $\CC^{2g}$, then the representation
$\VV_{\ud\mu}$ of $\Sp(2g,\ZZ)$ is the irreducible representation of highest weight that appears in the decomposition of the tensor product
$$
\Sym^{\mu_1-\mu_2}(\VV)\otimes \Sym^{\mu_2-\mu_3}\big(\bigwedge^2\VV\big) \otimes \dots\otimes \Sym^{\mu_{g-1}-\mu_g}\big(\bigwedge^{g-1}\VV\big)\otimes\big(\bigwedge^{g}\VV\big)^{\otimes\mu_g}
$$
into irreducible representations.
Any local system on $\ab$ defines a local system on $\ab(n)$ by pullback, or equivalently, by restricting a representation of $\Sp(2g,\ZZ)$ to the principal congruence group of level $n$ --- and we will denote this local system by the same symbol.
The local systems $\VV_{\ud\mu}$ are naturally a Hodge module of weight equal to the weight $w(\ud\mu)=\sum_{i=1}^{g}\mu_i$ of $\ud\mu$.
We will, however, not use these weights in what follows and thus not include them in our notation.

Recall that the cohomology ring of an abelian variety $A$ is
\begin{equation*}
H^*(A,\QQ)= \bigwedge H^1(A,\QQ).
\end{equation*}
Varying the abelian variety $A\in\calA_g$, the cohomology $H^1(A,\QQ)$ forms precisely the local system $\VV_1$.
Here we simply write $\VV_1$ rather than $\VV_{10\ldots 0}$ with $g-1$ zeroes. We will adapt a similar convention also for other local systems, dropping all the zero weights. This means that the same symbol can denote local systems on
different spaces $\ab$, but we hope that it will always be clear from the context on which space we are working. For future reference we recall
\begin{prop}\label{prop:univfamily}
The local systems associated to the universal family $\pi(n): \calX_g(n) \to \ab(n)$ are given, for $q\leq g$, by
\begin{equation}\label{equ:localuniversal}
 R^q\pi_*\QQ=
 \bigwedge^{q} \VV_1=\VV_{1^{q}}\oplus\VV_{1^{q-2}}\oplus\cdots\oplus\VV_{p(q)}
\end{equation}
where $p(q) \in \{0,1\}$ is the parity of $q$. The local systems in the range $q>g$ are given by duality.
\end{prop}
\begin{proof}
This follows immediately from the representation theory of the symplectic group, see also~\cite[Lemma 4.1]{GHT}.
\end{proof}
With these preparations we can now compute the (intersection) cohomology of the link bundles $\calN_{g-1,g}$.
\begin{prop}\label{prop:homlink}
For $q \leq g-1$ we have
\begin{equation*}
\mIH^q(\calN_{g-1,g})=\begin{cases}
                        \VV_{1^{q/2}} & \mbox{if $q$ is even}  \\
                        0 & \mbox{if $q$ is odd}.
                      \end{cases}
\end{equation*}
\end{prop}
\begin{proof}
We recall that on a level $n$ cover $\beta_1^{\Sigma,0}(n)=\calF_{g-1,g}(n)=\calX_{g-1}(n)$. In this case we have only one stratum $\beta_1^{\Sigma,0}(n)(\sigma)=\beta_1^{\Sigma,0}(n)$
as there is only one non-trivial cone in dimension $1$. The link bundle $\calN_{g-1,g}(n)$ is an $S^1$ circle bundle over $\calX_{g-1}(n)$ and we can use the associated Leray spectral sequence to
compute the (intersection) cohomology. At this point we can simplify matters by restricting to the invariant cohomology. This is very easy here: the elements of the stabilizer of type  $g_1$ and
$g_3$ (see~\cite[Section 7]{GHT} for this notation) act trivially on the cohomology of both $\calX_{g-1}(n)$ and the fiber $S^1$. On the other hand the element
$$
\begin{pmatrix}
A & B\\
C & D
\end{pmatrix} \in \Sp\left(2(g-i),\ZZ\right)
$$
with $B=C=0$, and $A=D$ the diagonal matrix with diagonal $(-1,\ldots,-1,1)$ acts by $-1$ on $\calX_{g-1}(n)$, and acts trivially on $S^1$.
Thus its action kills all of the odd degree cohomology of $\calX_{g-1}(n)$ and acts trivially on the rest (as it should, since we know we are eventually dealing with the universal Kummer family).
The Leray spectral sequence in the relevant range is thus given by Table~\ref{tab:cohlink1} where we have depicted the case that $g-1$ is even.
In case $g-1$ is odd, the last column is $0$.

\begin{table*}[!hbtp]
\scalebox{1.00}{
$
 \begin{array}{r|cccccccc}
 q&&&&&&&&\\
 1&\QQ&0&\QQ \oplus \VV_{11}&0 & \cdots &0& \QQ \oplus \VV_{11} \oplus \cdots \oplus \VV_{1^{g-1}}&\\
 0&\QQ&0&\QQ \oplus \VV_{11}&0 & \cdots &0& \QQ \oplus \VV_{11} \oplus \cdots \oplus \VV_{1^{g-1}}&\\
  \hline
  &0&1&2&3&\cdots & g-2&g-1&p\\
 \end{array}
 $
 }
 \caption{Leray spectral sequence for $\calN_{g-1,g}$}\label{tab:cohlink1}
\end{table*}

The differential in the Leray spectral sequence is given by mapping the generator of the fiber cohomology to the Euler class of the $S^1$ circle bundle.
Since the first Chern class of the normal bundle of the boundary in  $\Part$ is relatively ample, it follows that the
differentials $d_2^{pq}: E_2^{p,q} \to E_2^{p+2,q-1}$ are always of maximal rank. The claim now follows immediately from Table~\ref{tab:cohlink1}.
\end{proof}

\section{Computation of $IH^*(\Sat)$ for $g\le 2$}\label{sec:computations}
For $g=1$ the situation is trivial: $\Sat[1]\cong \PP^1$  is smooth
and thus  $IH^0(\Sat[1],\QQ) \cong IH^2(\Sat[1],\QQ) \cong \QQ$ and $IH^1(\Sat[1],\QQ)=0$. This can also
be rephrased by saying that $IH^*(\Sat[1],\QQ) \cong R_1$.

For  $g=2$  Mumford~\cite{Mu} determined the Chow ring of  $\overline \calM_2$. For the Chow ring of  $\oab[2]$ see~\cite{vdG}, and for the cohomology ring  see~\cite{HT}. The Betti numbers of $\oab[2]$   are given by the following table
\begin{equation*}
\begin{array}{r|ccccccc}
j&0&1&2&3&4&5&6 \\\hline
b_j=\dim H^j(\oab[2])&1&0&2&0&2&0&1
\end{array}
\end{equation*}
Recall from the decomposition theorem that $H^*(\oab[2] ,\QQ)$ is the direct sum of $IH^*(\Sat[2],\QQ)$ and some
(shifted) contributions from local systems on $\ab[0]$ and $\ab[1]$. However, note that the map $\phi_2: \oab[2] \to \Sat[2]$ is semi-small,
i.e. $r(\phi_2)=0$, and thus there are no non-trivial shifts.

We will now work out the details of the decomposition theorem and investigate which local systems appear. The first non-trivial step is to consider the stratum $\ab[1] \subset \Sat[2]$, by specializing the results of the previous section.
The following table shows what part of the cohomology we have accounted for so far from the intersection complex
$$
 \begin{array}{r|c|c|c}
 2&&&\mIH^2(\calN_{0,2})\\
 1&&\mIH^1(\calN_{1,2})=0& \mIH^1(\calN_{0,2})\\
 0&\QQ&\mIH^0(\calN_{1,2})=\QQ& \mIH^0(\calN_{0,2})=\QQ\\
 \hline
 &\ab[2]&\ab[1]& \ab[0]
 \end{array}
$$
This we must compare with the local systems given by the fiber cohomology, which is depicted in the following table
$$
 \begin{array}{r|c|c|c}
 2&&\calH^2(\calF_{1,2})=\underline{\QQ}&\calH^2(\calF_{0,2})=\QQ\\
 1&& \calH^1(\calF_{1,2})=0& \calH^1(\calF_{0,2})=0\\
 0& \QQ& \calH^0(\calF_{1,2})=\QQ &\calH^0(\calF_{0,2})=\QQ \\
 \hline
 &\ab[2]&\ab[1]& \ab[0]
 \end{array}
$$
Here the entries over $\ab[1]$ and $\ab[0]$ follow from the fact that $F_{1,2}$
is the Kummer variety of an elliptic curve (and thus a $\PP^1$), while $F_{0,2}$ is the quotient of a nodal $\PP^1$ by the involution, and thus also a $\PP^1$.
We now want to identify the new local systems that appear in the decomposition theorem. Over $\ab[1]$, we see that the only part of the cohomology of the fiber $\calH^*(\calF_{1,2})$ that has not been accounted for by $\mIH^*(\calN_{1,2})$ is $\calH^2(\calF_{1,2})=\QQ$, which is thus a new local system which appears in the decomposition theorem.
We have indicated this by underlining this system $\underline{\QQ}$, and in what follows we will always use this notation to indicate new local systems, i.e. those which have not been
accounted for by contributions from the bigger strata.
We also note that indeed the new local systems supported on $\ab[1]$ must be symmetric around degree 2 by Proposition~\ref{prop:symmetry}. On $\ab[0]$ this new local system $\underline{\QQ}$ on $\ab[1]$ contributes $\mIH^0(\calN_{0,1})=\QQ$ to the cohomology of $\calF_{0,2}$, accounting for $\calH^2(\calF_{0,2})=\QQ$. Thus there are no further new local systems supported on $\ab[0]$. Comparing the cohomology of the fiber with the link cohomology, and furthermore recalling that the intersection cohomology of $\calN_{0,2}$ satisfies Poincar\'e duality for degree 5, we obtain the results on the intersection cohomology of $\calN_{0,2}$ given in Theorem~\ref{thm:links}.

Summarizing, in genus 2 the decomposition theorem gives us simply
$$
 H^j(\oab[2])=IH^j(\Sat[2])\oplus IH^{j-2}(\Sat[1]).
$$
Since we know $IH^*(\Sat[1])=H^*(\Sat[1])=H^*(\PP^1)$, the knowledge of $H^*(\oab[2])$ yields
\begin{prop}\label{prop:IH2}
The Betti numbers  $ib_j:=\dim IH^j(\Sat[2],\QQ)$ are given by
\begin{equation} \label{equ:ibetti2}
\begin{array}{r|ccccccc}
j&0&1&2&3&4&5&6\\\hline
ib_j&1&0&1&0&1&0&1
\end{array}
\end{equation}
In particular there is an isomorphism $IH^*(\oab[2],\QQ)\cong R_2$ of graded vector spaces between the intersection cohomology and the tautological ring.
\end{prop}
\begin{rem}
Comparing with~\cite[Proposition 3]{Ha} we find that there is also an isomorphism
$$
IH^*(\Sat[2],\QQ) \cong H^2(\Sat[2],\QQ).
$$
As we shall see this is a coincidence that does not even extend to $g=3$.
\end{rem}

\section{Computation of $IH^*(\Sat[3])$}
We now move on to genus $3$. The main result is:
\begin{thm}\label{theo:IH3}
The intersection Betti numbers of $\Sat[3]$, denoted  $ib_j:=\dim IH^j(\Sat[3],\QQ)$ are given by the following table:
\begin{equation} \label{equ:ibetti3}
\begin{array}{r|ccccccccccccc}
j&     0&1&2&3&4&5&6&7&8&9&10&11&12\\\hline
ib_j&1&0&1&0&1&0&2&0&1&0&1 & 0  &1
\end{array}
\end{equation}
In particular, there is an isomorphism $IH^*(\Sat[3],\QQ)\cong R_3$ of graded vector spaces.
\end{thm}
\begin{rem}
Note that in this case $H^*(\Sat[3],\QQ)  \not\cong IH^*(\Sat[3],\QQ)$. Indeed, by~\cite[Theorem 2]{Ha} the middle cohomology $B:=H^6(\Sat[3],\QQ) $ is the 3-dimensional extension
$$
0 \to \QQ \to B \to 2\,\QQ(-3)\to 0.
$$
Since this is clearly not algebraic,  the map $H^*(\Sat[3],\QQ) \to IH^*(\Sat[3],\QQ)$ must have a non-trivial kernel (this map can easily be seen to be surjective, so that the kernel is the space of all non-algebraic classes).
\end{rem}
\begin{proof}
As before we really have to work on a level cover, but, using invariant cohomology in the spirit described above, we will formulate the entire discussion on   $\oab[3]$ itself. We shall again use the decomposition theorem, this time for the morphism $\phi_3:\oab[3]\to \Sat[3]$, and the fact that $\oab[3]$ is a smooth projective stack.
The cohomology $H^*(\oab[3],\QQ)$ was computed in~\cite{HT}, and along the way the cohomology of each fiber $F_{i,3}$, and of the fibration $\calF_{i,3}$ was also determined. We start our computations on $\calA_3\sqcup\calA_2\subset\Sat[3]$:
\begin{equation}
 \begin{array}{r|c|cl|cl}
 4&&&&\calH^4(\calF_{2,3})&=\underline{\QQ}\\
 3&&&&\calH^3(\calF_{2,3})&=0\\
 2&&\mIH^2(\calN_{2,3})&=\VV_{11}&\calH^2(\calF_{2,3})&=\underline{\QQ}\oplus\VV_{11}\\
 1&&\mIH^1(\calN_{2,3})&=0&\calH^1(\calF_{2,3})&=0\\
 0&\QQ&\mIH^0(\calN_{2,3})&=\QQ&\calH^0(\calF_{2,3})&=\QQ\\
 \hline
 & \calA_3&\multicolumn{2}{c|}{\hbox{Contributions on $\calA_2$}}& \calH^*\, (\calF_{2,3})
 \end{array}
\end{equation}
The first column of this table is simply the intersection cohomology complex on $\ab[3]$, and the second column
gives its contribution on $\ab[2]$, whereas the last column lists the local systems that form  the fiber cohomology of the fibration $\calF_{2,3}$.
The cohomology $\mIH^*(\calN_{2,3})$
is given by Proposition~\ref{prop:homlink}, and contributes in degree up to two, while $\calH^*(\calF_{2,3})$ follows from Proposition~\ref{prop:univfamily} (and also was determined in \cite{HT}).

As a result, we can identify the new local systems supported on $\ab[2]$, which are underlined in the table above. Note that they are indeed symmetric around degree 3, as they should be by Proposition~\ref{prop:symmetry}.

We proceed to work over $\ab[1]\subset\Sat[3]$; the cohomology $\calH^*(\calF_{1,3})$ was computed, as local systems over $\ab[1]$, in the table in~\cite[Lemma~8]{HT}. In fact, all that we need from these computations is that $\calH^6(\calF_{1,3})=\QQ$ and that $\calH^5(\calF_{1,3})=0$.
The situation is summarized as follows
\begin{equation}\label{equ:contrA1}
 \begin{array}{r|c|c|ccc|l}
 6&&&&&&\underline{\QQ}\\
 5&&&&&\mIH^1(\calN_{1,2})&0\\
 4&&\QQ&\mIH^4(\calN_{1,3})&&\mIH^0(\calN_{1,2})&\underline{\QQ}\oplus \QQ\\
 3&&&\mIH^3(\calN_{1,3})&\mIH^1(\calN_{1,2})&&0\\
 2&&\QQ\hskip1cm&\mIH^2(\calN_{1,3})&\mIH^0(\calN_{1,2})&&\QQ\\
 1&&&\mIH^1(\calN_{1,3})&&&0\\
 0&\QQ&&\mIH^0(\calN_{1,3})&&&\QQ\\
 \hline
 &\calA_3& \calA_2&\multicolumn{3}{c|}{\hbox{Contributions on $\calA_1$}}&\calH^*(\calF_{2,3})
 \end{array}
\end{equation}
Here the first column shows that over $\ab[3]$ the only new local system is $\QQ$ in degree zero. Over $\ab[2]$ there are two new local systems, each isomorphic to $\QQ$, in degrees 2 and 4 respectively. The columns above $\ab[1]$ then indicate the contributions to the cohomology of the fibers over $\ab[1]$ of these three new local systems on $\ab[3]$ and $\ab[2]$, with each column giving the contributions of one new local system $\QQ$, Finally, the last column gives the local systems that form the cohomology of the fibration $\calF_{2,3}$.
From this we see that the top degree cohomology of the fiber $\calH^6(\calF_{2,3})=\QQ$ cannot be accounted for by the intersection homology of the links, and thus gives a new local system $\QQ[-1]$ on $\ab[1]$. By symmetry of new local systems around degree 5 provided by Proposition~\ref{prop:symmetry}, the only other new local system on $\ab[1]$ must then be $\QQ[1]$; following our previous convention, we have underlined these two new local systems on $\ab[1]$.
Summarizing the contributions from the new local systems that we have found so far, we thus get
\begin{equation}\label{eq:containgenus3}
\begin{aligned}
 H^*(\oab[3])\supseteq
  IH^*(\Sat[3])\oplus IH^*(\Sat[2])[-1]\oplus IH^*(\Sat[2])[1] \\
  \oplus IH^*(\Sat[1])[-1]\oplus IH^*(\Sat[1])[1].
  \end{aligned}
\end{equation}
All the cohomology above except for $IH^*(\Sat[3])$ is known. Indeed, the cohomology of $\oab[3]$ was computed by Hain~\cite{Ha}, while the intersection cohomology of $\Sat[2]$ we computed in the previous section. Denoting $r^j:=\dim R^j_3$ the dimension of the graded pieces of the tautological ring, which by Proposition~\ref{prop:tautological} is contained in $IH^*(\Sat[3])$, the above inclusion gives Table~\ref{equ:ih3}.

\begin{table*}[!hbtp]
\scalebox{1.00}{
$
\begin{array}{rl|rrrrrrrrrrrrr}
j& &\ 0 &\ 1&\ 2 &\ 3 &\ 4 &\ 5 &\ 6 &\ 7 &\ 8 &\ 9 & 10 & 11 & 12 \\
\hline 
\rule{0pt}{4mm}    h^j(\oab[3])&{} &1  &0 &2  &0  &4  &0  &6  & 0 & 4 &0 &2 &0 &1\\\hdashline 
\rule{0pt}{4mm}    &\multicolumn{12}{c}{\hbox{ must contain the direct sum of:}}\\\hdashline 
\rule{0pt}{4mm}    r^j(\Sat[3])& &1  &0 &1  &0  & 1  &0  &2  & 0 & 1 &0 & 1 &0 &1\\
ih^{j}(\Sat[2])&[-1] &  &  & 1  & 0 & 1 & 0  & 1 & 0 & 1 &    &        &  &\\
ih^{j}(\Sat[2])&\!\!\!\![1] &  &  &         &    & 1  &  0  & 1  & 0  &  1 & 0 & 1 &   &    \\
ih^{j}(\Sat[1])&\!\!\!\![-1] &  &  &         &    & 1  & 0  & 1 &   &   &    &        &  &\\
ih^{j}(\Sat[1])&\!\!\!\![1] &  &  &         &    &        &      & 1  & 0  & 1 &  & &   &    \\
\end{array}
$
}
 \caption{Decomposition theorem in genus $3$}\label{equ:ih3}
\end{table*}

Observe now that in Table~\ref{equ:ih3} the sum of the five bottom rows, i.e.~of the individual contributions to the cohomology of $\oab[3]$, is precisely equal to the top row. It thus follows that the containment in \eqref{eq:containgenus3} is in fact an equality. Furthermore, it follows that $IH^*(\Sat[3])=R^*(\Sat[3])$, proving the theorem.
\end{proof}
Combining \eqref{equ:contrA1} with the fact that the intersection cohomology of $\calN_{1,3}$ satisfies Poincar\'e duality in degree 9 gives all of $\mIH^*(\calN_{1,3})$ as stated in  Theorem~\ref{thm:links}.

Note that the above computations did not require us to compute the cohomology of $\calF_{0,3}$, as we could already ascertain that there can be no further new local systems with non-trivial cohomology and that \eqref{eq:containgenus3} is indeed an equality. However, for what follows in genus 4, and for the record, we would like to summarize what happens over $\ab[0]$. This allows us to compute the cohomology of $\calN_{0,3}$.

Indeed, the cohomology of $\calF_{0,3}$ was computed in~\cite[Proposition 2.4]{HT} --- and in particular, is zero in all odd degrees. The link $\calN_{0,3}$ contributes its intersection cohomology in degree up to 5; the link $\calN_{0,2}$ contributes its cohomology a priori in degree up to 2, but, as we proved in the previous section, its only non-zero cohomology is in degree zero. Note, however, that over $\ab[2]$ we have two new local systems, in degrees 2 and 4. Finally, the link $\calN_{0,1}$ only contributes its $\mIH^0(\calN_{0,1})=\QQ$, for the two new local systems in degrees 4 and 6. Writing down the contributions to $\calH^*(\calF_{0,3})$ from the new local systems on the bigger strata we thus get the following table:
$$
\scalebox{0.95}{$
 \begin{array}{r|r|r|r|r}
 6&&\mIH^2(\calN_{0,2})=0&\mIH^0(\calN_{0,1})=\QQ&\QQ\\
 4&\mIH^4(\calN_{0,3})&\mIH^0(\calN_{0,2})\oplus\mIH^2(\calN_{0,2})=\QQ&\mIH^0(\calN_{0,1})=\QQ&2\,\QQ\\
 2&\mIH^2(\calN_{0,3})&\mIH^0(\calN_{0,2})=\QQ&&\QQ\\
 0&\mIH^0(\calN_{0,3})&&&\QQ\\
 \hline
 &\mbox{from }\ab[3]&\mbox{from }\ab[2]&\mbox{from }\ab[1]& \calH^*(\calF_{0,3})
 \end{array}
 $
 }
$$
We have omitted to mention the cohomology in odd degree, as it always vanishes.
Since there are no new local systems supported on $\ab[0]$, all of the cohomology of the fiber must have been precisely accounted for by the intersection cohomology of the links. Combining this with Poincar\'e duality computes all of $\mIH^*(\calN_{0,3})$ as given in Theorem~\ref{thm:links}.

\section{Computations of intersection homology in genus 4}
In this section we use the decomposition theorem to compute all of the intersection cohomology of $\Sat[4]$ and of $\Perf[4]$, except in degree 10, and also to get lower bounds in degree 10. The machinery is an extension of that used in the previous sections, and uses the computation of the cohomology of all the fibers of the map $\Vor[4]\to\Sat[4]$ done in~\cite{HT2}, while many of the needed links have already been studied in the situations of lower genera.

\smallskip
We first deal with $\Perf[4]$. Recall that  the singular locus of $\Perf[4]$ consists of one point, and the map $\Vor[4]\to\Perf[4]$ is the blow-up of this singular point, with exceptional divisor $E\subset\Vor[4]$, see~\cite{EGH,HS}. Indeed, the exceptional divisor $E$ is smooth, but we will not need this. The intersection homology of the blow-up of a point is well-known, and easily
follows from our setup:
\begin{lm}\label{lm:Perf4}
$H^j(\Vor[4])\cong \begin{cases}
  IH^j(\Perf[4]), & \mbox{if } 19\le j\le 20 \\
  IH^j(\Perf[4]) \oplus H^j(E), & \mbox{if } 10\le j\le 18 \\
  IH^j(\Perf[4])\oplus H^{20-j}(E), & \mbox{if } 2\le j\le 9\\
  IH^j(\Perf[4]), & \mbox{if } 0\le j\le 1.
\end{cases}$
\end{lm}
\begin{proof}
In our setup, the stratification of $\Perf[4]$ consists of the singular point and its complement. On the complement we have the local system $\QQ$, and since the
complex codimension of the singular point is $10$, the link contributes to the fiber cohomology up to degree 9.
On the other hand, the new local systems that need to account for the cohomology of the fiber $H^*(E)$ must be symmetric around degree 10 by Proposition~\ref{prop:symmetry}. Since in degree 10 and higher all of $H^*(E)$ is new, the new local systems in degree $j\le 9$ must then be $H^{20-j}(E)$ by  symmetry. Of course, since $\dim_\CC E=9$, there is no contribution in degrees $0,1,19,20$.
\end{proof}
Using this and the computations of $H^*(\Vor[4])$ and $H^*(E)$, as done in ~\cite[Theorem 1]{HT2} and~\cite[Theorem 26(1)]{HT2}, we obtain all claims in Proposition~\ref{prop:Perf4}, apart from the claim in degree $10$. Further below we will see that $b_{10}(\Vor[4]) \geq 19 $, which, together with $b_{10}(E)=3$  then gives us $ib_{10}(\Perf[4]) \geq 16$, as claimed.

\smallskip
We now proceed to compute $IH^*(\Sat[4])$ by investigating the decomposition theorem for the morphism $\phi_4:\Vor[4]\to\Sat[4]$. As before, we first do the computation over $\calA_3\subset\Sat[4]$, where the relevant cohomology of the fibration $\calF_{3,4}$ is given by Proposition~\ref{prop:univfamily} and the (intersection) cohomology of the link $\calN_{3,4}$ is given by Proposition~\ref{prop:homlink}. Again, we have underlined the new local systems.
At this point we would like to remark that the degree $6$ entry is a priori $\VV^{\vee}_{11}$. However, all symplectic representations are self-dual, and since we are omitting all Hodge weights, we can simply write $\VV_{11}$.
$$
 \begin{array}{r|c|cl|cl}
 6&&&&\calH^6(\calF_{3,4})&=\underline{\QQ}\\
 5&&&&\calH^5(\calF_{3,4})&=0\\
 4&&&&\calH^4(\calF_{3,4})&=\underline{\QQ}\oplus\underline{\VV_{11}}\\
 3&&\mIH^3(\calN_{3,4})&=0&\calH^3(\calF_{3,4})&=0\\
 2&&\mIH^2(\calN_{3,4})&=\VV_{11}&\calH^2(\calF_{3,4})&=\underline{\QQ}\oplus\VV_{11}\\
 1&&\mIH^1(\calN_{3,4})&=0&\calH^1(\calF_{3,4})&=0\\
 0&\QQ&\mIH^0(\calN_{3,4})&=\QQ&\calH^0(\calF_{3,4})&=\QQ\\
 \hline
 &{\rm on\ }\calA_4&{\rm on\ }\calA_3&&\calH^*(\calF_{3,4})
 \end{array}
$$
\begin{rem}\label{rem:V11}
This is the first case where we have a non-trivial (not $\QQ$) new local system appearing in the decomposition theorem, namely  $\VV_{11}$,
It will follow from our computations, however, that $IH^k(\Sat[3],\VV_{11})=0$,  except possibly in the middle degree $k=6$ -- since all the rest of the intersection cohomology of $\Vor[4]$ will already be accounted for. We see no a priori reason to expect this.
\end{rem}
We now proceed to work over $\ab[2]\subset\Sat[4]$, where the cohomology of the fiber $\calF_{2,4}$ is given in the left-most column in~\cite[Table 8]{HT2} (where it appears as $R_!^q(k_2\circ k_3)_*(\QQ)$).
We note that while at this point we do not know the intersection cohomology $\mIH^*(\calN_{2,4})$, we already know from Proposition~\ref{prop:homlink}
that $\mIH^0(\calN_{2,3})=\QQ$, $\mIH^1(\calN_{2,3})=0$ and $\mIH^2(\calN_{2,3})=\VV_{11}$.
Substituting these, and denoting $*^i:=\mIH^i(\calN_{2,3},\VV_{11})$, the computation can be summarized as follows:
\begin{equation}\label{equ:contribA2}
 \begin{array}{r|r|r|l|r}
 10&&&&\underline{\QQ}\\
 9&&&&0\\
 8&&&\VV_{11}&\VV_{11}\oplus\underline{2\,\QQ}\\
 7&&&&0\\
 6&&\QQ&\mIH^6(\calN_{2,4})\oplus*^2\oplus\VV_{11}\oplus\QQ&\VV_{22}\oplus\VV_{11}\oplus\QQ\oplus\underline{2\, \QQ}\\
 5&&&\mIH^5(\calN_{2,4})\oplus*^1&\VV_2\\
 4&&\QQ\oplus\VV_{11}&\mIH^4(\calN_{2,4})\oplus\VV_{11}\oplus\QQ\oplus *^0 &\VV_{22}\oplus\VV_{11}\oplus\QQ\oplus\underline{\QQ}\\
 3&&&\mIH^3(\calN_{2,4})&0\\
 2&&\QQ&\mIH^2(\calN_{2,4})\oplus\QQ&\QQ\\
 1&&&\mIH^1(\calN_{2,4})&0\\
 0&\QQ&&\mIH^0(\calN_{2,4})&\QQ\\
 \hline
 &\calA_4&\calA_3&\mbox{All contributions over }\calA_2&\calH^*(\calF_{2,4})
 \end{array}
\end{equation}
Since the new local systems are symmetric around degree 7 by Proposition~\ref{prop:symmetry}, the only piece of information about intersection cohomology of the links that we actually use to determine all new local systems  is that $\mIH^2(\calN_{2,3})=\VV_{11}$.
As a result, the new local systems supported on $\ab[2]$ are $\QQ[-3]\oplus 2\,\QQ[-1]\oplus2\,\QQ[1]\oplus \QQ[3]$.
Moreover, we can relate the cohomology of $\calN_{2,4}$ and the cohomology of $\calN_{2,3}$ with coefficients in $\VV_{11}$, as stated in Theorem~\ref{thm:links}.

Proceeding to work over $\ab[1]\subset\Sat[4]$, we need to know the cohomology of the fiber $\calF_{1,4}$.
In principle, all the necessary computations were done in~\cite[Sec.~6]{HT2}. There is, however, one difference. In ~\cite{HT2} the Gysin sequence was applied to the strata $\beta(\sigma)$ themselves,
where $\sigma$ runs through all $\Sp(6,\ZZ)$-orbits of cones in $\Sym^2_{\geq 0}(\RR^3)$ containing matrices of full rank. Here, in contrast, we want to apply the Gysin sequence fiberwise  to
the fibrations $\beta(\sigma) \to \ab[1]$.  Nevertheless, we will see that all required information has already been obtained in~\cite{HT2}, and it allows us to prove the following proposition.
\begin{prop}\label{prop:HF14}
The cohomology groups of the fibration $\calF_{1,4}$, as local systems on $\ab[1]$, vanish in every  odd degree, while the even degree ones are as follows:
\begin{equation} \label{equ:betti14}
\begin{array}{r|ccccccc}
j&0&2&4&6&8&10&12\\\hline
\rule{0pt}{4mm}    \calH^j(\calF_{1,4})&\QQ&\QQ&\VV_2\oplus 2\QQ&\VV_2\oplus 4\QQ&4\QQ&3\QQ&\QQ\\
\end{array}
\end{equation}
\end{prop}
\begin{proof}
We recall that the stratum $\beta_3^0$ of the map $\phi_4: \overline{\calA}_4 \to \Sat[4]$ over $\ab[1]$ consists of five strata $\beta(\sigma)$ of the toroidal compactification, indexed by the cones $\sigma^{(3)},\sigma_I^{(4)},\sigma_{II}^{(4)},\sigma^{(5)},\sigma^{(6)}$ in~\cite{HT2}, which were relabeled in a more
systematic manner $\sigma_{K_3}$, $\sigma_{K_3+1}$, $\sigma_{C_4}$, $\sigma_{K_4-1}$ and $\sigma_{K_4}$ in~\cite{GHT}.
For each such cone $\sigma$ we denote, as in~\cite{HT2}, the fiber of $\beta(\sigma) \to \ab[1]$ by $F^{(3)},F_I^{(4)},F_{II}^{(4)},F^{(5)}$ and $F^{(6)}$ respectively, which have
dimension  $6,5,5,4,$ and $3$, and we will denote by $\calF$ the corresponding fibrations.
We list the cohomology with compact support for the various fibrations in~\eqref{equ:cohfiber3} below. The entries   for $\calH^*_c\calF^{(3)}$ and $\calH^*_c\calF_I^{(4)}$ are given in~\cite[Lemma 17]{HT2} and~\cite[Lemma 19]{HT2} respectively.
The cohomology with compact support of the total spaces of the other three strata over $\ab[1]$ was computed in~\cite[Lemmas~21,22,23]{HT2}, respectively. The proof of each of these lemmas proceeds by first computing the cohomology with compact support of the corresponding fibration $\calF$, as local systems over $\ab[1]$, and then computing the cohomology with compact support of these local systems on $\ab[1]$. For each of these fibrations, all the local systems on $\ab[1]$ that appear in the cohomology with compact support are equal to $\QQ$, and its only non-trivial cohomology with compact support is $\calH^2_c(\ab[1],\QQ)=\QQ$. Thus the cohomology of these fibers in any degree $ $is equal to the cohomology of the total spaces of the corresponding strata in degree $d+2$.
This immediately gives the values in (\ref{equ:cohfiber3}),
where we have additionally included the respective Hodge weights. This serves as a check because we will have to investigate a certain differential in the Gysin spectral sequence.

\begin{equation}
 \begin{array}{r|rrrrr}
 12&\QQ(-6)\\
 11\\
 10&\QQ(-5)&\QQ(-5)&\QQ(-5)\\
 9\\
 8&\VV_2(-2)&2\QQ(-4)&\QQ(-4)&\QQ(-4)\\
 7&&\VV_2(-2)\\
 6&\VV_2(-1)&\QQ(-3)&&2\QQ(-3)&\QQ(-3)\\
 5\\
 4&&\VV_2(0)\ &&\QQ(-2)&\QQ(-2)\\
 3\\
 2&&&&&\QQ(-1)\\
 1\\
 0&&&&&\QQ(0)\ \\
 \hline
 \rule{0pt}{4mm}    &\calH^*_c(\calF^{(3)})& \calH^*_c(\calF_I^{(4)})&\calH^*_c(\calF_{II}^{(4)})&\calH^*_c(\calF^{(5)})&\calH^*_c(\calF^{(6)})
 \end{array}
 \label{equ:cohfiber3}
\end{equation}
We now compute the cohomology with compact support of the fibration $\calF_{1,4}$, which is the union of the 5 pieces above. To do this, we use a Gysin spectral sequence (of local systems)
and apply this to the stratification of $\calF_{1.4}$ given by $\calF^{(6)}$, $\calF_I^{(5)}$, $\calF_I^{(4)} \cup \calF_{II}^{(4)}$ and $\calF^{(3)}$, all considered as fibrations over $\ab[1]$, so that the relative cohomology forms local systems on $\ab[1]$.
Its $E_1$ page is given by (\ref{tab:Gysin}) --- compare to~\cite[Table~13]{HT2} for the Gysin spectral sequence for $H^*_c(\beta_3^0)$:

\begin{equation}
 \begin{array}{r|ccccl}
 q\\
 9&&&&\QQ(-6)\\
 8&&&2\QQ(-5)&\\
 7&&\QQ(-4)&&\QQ(-5)\\
 6&\QQ(-3)&&3\QQ(-4)&\\
 5&&2\QQ(-3)&\VV_2(-2)&\VV_2(-2)\\
 4&\QQ(-2)&&\QQ(-3)&\\
 3&&\QQ(-2)&&\VV_2(-1)\\
 2&\QQ(-1)&&\VV_2(0)&\\
 1&&&&\\
 0&\QQ(0)&&&\\
 \hline
 &0&1&2&3&p\\
 &\calH^*_c(\calF^{(6)})& \calH^*_c(\calF^{(5)})   &\calH^*_c(\calF_I^{(4)}) \oplus \calH^*_c(\calF_{II}^{(4)})&\calH^*_c(\calF^{(3)})
 \end{array}
 \label{tab:Gysin}
\end{equation}

The only possibly non-trivial differential in the Gysin spectral sequence for $\calH^*(\calF_{1,4})$ is thus $d_1^{2,5}:\VV_2(-2)\to\VV_2(-2)$.
We claim that this is non-trivial and hence an isomorphism.  Indeed,  this is equivalent to saying that the fiber $F_{1,4}$ has no cohomology in degree $7$, or equivalently
that the differential $d_7$
in the proof of ~\cite[Prop.~24]{HT2} is an isomorphism, which was shown to be the case on ~\cite[p.~230]{HT2}.
Thus the total cohomology with compact support of the fibration $\calF_{1,4}$ is the sum of all the other entries, besides the two copies of $\VV_2(-2)$, in the Gysin spectral sequence above. As the fiber
$F_{1,4}$ is compact, its cohomology with compact support is equal to its cohomology, and is precisely as stated in the proposition.
\end{proof}
We now return to implementing our method to find the new local systems supported on $\ab[1]$ for the decomposition theorem for the map $\phi_4:\Vor[4]\to\Sat[4]$.
Since all the odd degree cohomology of $\calF_{1,4}$ is zero, to simplify the tables we only write down the even degree contributions from all local systems.
Further using the fact that the only non-zero intersection homology $\mIH^j(\calN_{1,2})$ and $\mIH^j(\calN_{1,3})$
are in degree zero, by Proposition~\ref{prop:homlink} and by the part of Theorem~\ref{thm:links} on $\calN_{1,3}$ that we already proved in the section with genus 3 computations, we obtain
\begin{equation*}
\scalebox{0.88}{
$
 \begin{array}{r|r|r|r|l|r}
 12&&&&&\underline{\QQ}\\
 10&&&\QQ&\QQ&\QQ\oplus\underline{2\,\QQ}\\
 8&&&2\,\QQ&\mIH^8(\calN_{1,4})\oplus\mIH^4(\calN_{1,3},\VV_{11})\oplus 2\,\QQ&2\,\QQ\oplus\underline{2\,\QQ}\\
 6&&\QQ&2\,\QQ&\mIH^6(\calN_{1,4})\oplus\mIH^2(\calN_{1,3},\VV_{11})\oplus 3\,\QQ&\VV_2\oplus 3\,\QQ\oplus\underline{\QQ}\\
 4&&\VV_{11}\oplus\QQ &\QQ &\mIH^4(\calN_{1,4})\oplus\mIH^0(\calN_{1,3},\VV_{11})\oplus2\,\QQ&\VV_{2}\oplus 2\,\QQ\\
 2&&\QQ&&\mIH^2(\calN_{1,4})\oplus\QQ&\, \ \QQ\\
 0&\QQ&&&\mIH^0(\calN_{1,4})&\QQ\\
 \hline
 &\calA_4&\calA_3&\calA_2&\mbox{Total of prior contributions over }\ab[1]& \calH^*(\calF_{1,4})
 \end{array}
 $
 }
\end{equation*}

By Proposition~\ref{prop:symmetry} the new local systems are symmetric around degree 9, and thus they must be
$\QQ[-3]\oplus 2\,\QQ[-1]\oplus 2\,\QQ[1]\oplus\QQ[3]$. As before, matching the contributions over $\ab[1]$ from the new local systems on bigger strata to the local systems in the cohomology of $\calF_{1,4}$ which are not new, allows us to deduce the results on $\mIH^*(\calN_{1,4},\QQ)$ and $\mIH^*(\calN_{1,3},\VV_{11})$ as stated in Theorem~\ref{thm:links}.

We now need to deal with the fibration $\calF_{0,4}$ over $\calA_0$ (which is simply the fiber $F_{0,4}$ over one point). This is the union of $E$ (the exceptional divisor of the contraction $\overline\ab[4]\to\Perf[4]$) and the strict transform of the fiber of the map $\Perf[4]\to\Sat[4]$ over $\ab[0]$ ,
where the latter was denoted $\beta_4^{\op{Perf}}$ in~\cite{HT2}.
The cohomology of the fiber  $F_{0,4}$  is given by~\cite[Theorem 26 (2)]{HT2}, where this fiber was denoted $\beta_4$: the only non-zero Betti numbers are the even degree ones, and they are equal to
\begin{equation} \label{equ:F04}
\begin{array}{r|cccccccccc}
j&0&2&4&6&8&10&12&14&16&18\\\hline
b_j(F_{0,4})&1&2&3&7&7&6&4&2&1&1
\end{array}
\end{equation}
Similarly to the proof of Lemma~\ref{lm:Perf4} above, the new local systems appearing must be symmetric around degree 10 by Proposition~\ref{prop:symmetry}. We thus first focus on understanding what part of the cohomology of $F_{0,4}$ in degree 10 and higher can be accounted for by local systems on previous strata, and then use the symmetry to compute some intersection cohomology of the relevant links.

We have already proven the relevant statements in Theorem~\ref{thm:links}, that the only non-trivial intersection cohomology of $\calN_{0,2}$ and $\calN_{0,3}$ with $\QQ$ coefficients are in degree zero.

The contributions towards $\calH^*(F_{0,4})$ from the new local systems over the other strata are summarized in the following table --- where we omit all the odd degrees since the cohomology of $F_{0,4}$ is then zero.
Again, we can easily find the new (underlined)  local systems in degree 10 and higher; by the symmetry given by Proposition~\ref{prop:symmetry} this also gives us the new local systems in lower degrees.
$$
 \begin{array}{r|r|r|r|r|r}
 18&&&&&\underline{\QQ}\\
 16&&&&&\underline{\QQ}\\
 14&&&&&\underline{2\,\QQ}\\
 12&&&&\QQ&\QQ\oplus\underline{3\,\QQ}\\
 10&&&\QQ&2\,\QQ&3\,\QQ\oplus\underline{3\,\QQ}\\
 8&\mIH^8(\calN_{0,4})&\mIH^4(\calN_{0,3},\VV_{11})\qquad\qquad &2\,\QQ&2\,\QQ&4\,\QQ\oplus\underline{3\,\QQ}\\
 6&\mIH^6(\calN_{0,4})&\mIH^2(\calN_{0,3},\VV_{11})\oplus\QQ&2\,\QQ&\QQ&5\,\QQ\oplus\underline{2\,\QQ}\\
 4&\mIH^4(\calN_{0,4})&\mIH^0(\calN_{0,3},\VV_{11})\oplus\QQ&\QQ&&2\,\QQ\oplus\underline{\QQ}\\
 2&\mIH^2(\calN_{0,4})&\QQ&&&\QQ\oplus\underline{\QQ}\\
 0&\mIH^0(\calN_{0,4})&&&&\QQ\\
 \hline
 \rule{0pt}{4mm}    &\mbox{from }\ab[4]&\mbox{from }\ab[3]&\mbox{from }\ab[2]&\mbox{from }\ab[1]& \calH^*(F_{0,4})
 \end{array}
$$

Since all the cohomology $\calH^*(F_{0,4})$ must be accounted for by the new local systems and contributions from bigger strata, we get the statements about the cohomology of the link $\calN_{0,4}$ as given in Theorem~\ref{thm:links}

Assembling the contributions we have accounted for, we have
$$
 \begin{aligned}H^*&(\Vor[4])= IH^*(\Sat[4])\!
 \oplus\! IH^*(\Sat[3])[-2]\!\oplus\! IH^*(\Sat[3])[0]\!\oplus IH^*(\Sat[3])[2]\\ \oplus & IH^*(\Sat[2])[-3]\oplus 2\, IH^*(\Sat[2])[-1]\oplus 2\,IH^*(\Sat[2])[1]\oplus IH^*(\Sat[2])[3]\\
 \oplus  &IH^*(\Sat[1])[-3]\oplus 2\, IH^*(\Sat[1])[-1]\oplus 2\, IH^*(\Sat[1])[1]\oplus IH^* (\Sat[1])[3]\\ \oplus &(\mbox{new part of }\calH^*(F_{0,4}), \mbox{ symmetrically})[0]\oplus IH^*(\Sat[3],\VV_{11})[0].
 \end{aligned}
$$
We have no  a priori information about the last summand, but from the computation below we will conclude from this that its only possibly non-zero intersection cohomology group is $IH^6(\Sat[3],\VV_{11})$.
Otherwise, all the intersection cohomology on the right is known from the previous sections, except for
$IH^*(\Sat[4])$, which it is our aim to determine. Our computations are summarized in Table~\ref{equ:ih4},
where we have again only indicated the even degrees, as all odd degree intersection cohomology groups involved are zero.

\begin{table*}[!hbtp]
$\begin{array}{rl|rrrrrrrrrrr}
&j &\ 0 &\ 2&\ 4 &\ 6 &\ 8 & 10 & 12 & 14 &16 & 18 & 20 \\
\hline 
\rule{0pt}{4mm}     ih^j(\Vor[4])&&1  &3 &5  &11  &17  &?  &17  & 11 & 5 &3 &1\\\hdashline
\rule{0pt}{4mm}      &\multicolumn{10}{c}{\hbox{ must contain the direct sum of}}\\\hdashline 
\rule{0pt}{4mm}       
r^j(\Sat[4])&&1  &1 &1  &2  & 2  &2  &2  & 2 & 1 &1 & 1\\
ih^{j}(\Sat[3])&\!\!\!\![-2]&&1&1&1&2&1&1&1\\
ih^{j}(\Sat[3])&\!\!\!\![0]&&&1&1&1&2&1&1&1\\
ih^{j}(\Sat[3])&\!\!\!\![2]&&&&1&1&1&2&1&1&1\\
ih^{j}(\Sat[2])&\!\!\!\![-3]&&&1&1&1&1\\
2\,ih^{j}(\Sat[2])&\!\!\!\![-1]&&&&2&2&2&2\\
2\,ih^{j}(\Sat[2])&\!\!\!\![1]&&&&&2&2&2&2\\
ih^{j}(\Sat[2])&\!\!\!\![3]&&&&&&1&1&1&1\\
ih^{j}(\Sat[1])&\!\!\!\![-3]&&&&1&1 \\
2\,ih^{j}(\Sat[1])&\!\!\!\![-1]&&&&&2&2\\
2\,ih^{j}(\Sat[1])&\!\!\!\![1]&&&&&&2&2\\
ih^{j}(\Sat[1])&\!\!\!\![3] &&&&&&&1&1\\
\mbox{From }\calH^*(F_{0,4})&&&1&1&2&3&3&3&2&1&1&\\
\hdashline
\rule{0pt}{4mm}     
{\rm Sum\ equals}&&1&3&5&11&17&19&17&11&5&3&1
\end{array}$
 \caption{Decomposition theorem in genus $4$}\label{equ:ih4}
\end{table*}

As a result, we see that in every degree except possibly 10 the summands above account for all of the intersection cohomology of $\Vor[4]$, and thus we must have $IH^*(\Sat[4])=R^*(\Sat[4])$, except possibly for degree 10. This finally yields our main result in genus 4, as well as the inequality $h^{10}(\Vor[4])\ge 19$, which by Lemma~\ref{lm:Perf4} also gives $ih^{10}(\Perf[4])\ge 16$, thus  finishing the proof of Proposition~\ref{prop:Perf4}. To summarize, we have proven

\begin{thm}\label{theo:IH4}
The intersection cohomology of $\Sat[4]$ is equal to the tautological ring of $\Sat[4]$, except it may possibly be bigger in degree 10. That is, we have for the even degree intersection Betti numbers:
\begin{equation} \label{equ:ibetti4}
\begin{array}{r|ccccccccccc}
j&     0&2&4&6&8&10&12&14&16&18&20\\\hline
\rule{0pt}{4mm}      ib_j(\Sat[4])&1&1&1&2&2&\ge 2&2&2&1&1&1
\end{array}
\end{equation}
while all the odd intersection homology groups of $\Sat[4]$ are zero.
\end{thm}

\begin{rem}\label{rem: noA0}
In our computations for $g\le 4$ we have found that there are no local systems supported on $\ab[0]$ contributing to the decomposition theorem. We expect this
to be true for any genus $g$.
\end{rem}

\end{document}